\documentclass[letterpaper,10pt]{article}

\usepackage{opticameet3} 

\newcommand\authormark[1]{\textsuperscript{#1}}

\usepackage{multirow}
\usepackage{multicol}
\usepackage{amsfonts,booktabs}
\usepackage{pifont}
\usepackage{amsthm}
\usepackage{amsmath}
\usepackage{cases}
\usepackage{amssymb}
\usepackage{graphicx}
\usepackage{subcaption}
\usepackage{url}
\usepackage[justification=centering,labelsep=period]{caption}
\newtheorem{theorem}{Theorem}[section]
\newtheorem{lemma}{Lemma}[section]
\newtheorem{remark}{Remark}[section]
\newtheorem{definition}{Definition}[section]

\numberwithin{equation}{section}
\numberwithin{figure}{section}
\numberwithin{table}{section}
\usepackage[colorlinks=true,bookmarks=false,citecolor=blue,urlcolor=blue]{hyperref} 
\DeclareMathAlphabet{\mathcal}{OMS}{cmsy}{m}{n}

\begin{document}

\title{Asymptotic-preserving conservative semi-Lagrangian discontinuous Galerkin schemes for the Vlasov-Poisson system in the quasi-neutral limit}


\author{Xiaofeng Cai\authormark{1,2},   Linghui Kong\authormark{3,*}, Dmitri Kuzmin\authormark{4}, and Li Shan\authormark{5}}

\address{\authormark{1}Department of Mathematics, Faculty of Arts and Sciences, Beijing Normal University, Zhuhai, 519087, P.R. China (xfcai@bnu.edu.cn).\\
\authormark{2}Guangdong Provincial/Zhuhai Key Laboratory of Interdisciplinary Research and Application for Data Science, Beijing Normal-Hong Kong Baptist University, Zhuhai, 519087, P.R. China.\\
\authormark{3}Department of Mathematics, Faculty of Arts and Sciences, Beijing Normal University, Zhuhai, 519087, P.R. China (linghuikong@mail.bnu.edu.cn).\\
\authormark{4}Institute of Applied Mathematics (LS III), TU Dortmund University Vogelpothsweg 87, D-44227 Dortmund, Germany (kuzmin@math.tu-dortmund.de). \\ 
\authormark{5}Department of Mathematics, Shantou University, Shantou 515063, P.R. China (lishan@stu.edu.cn).}

\email{\authormark{*}Corresponding author at: Department of Mathematics, Faculty of Arts and Sciences, Beijing Normal University, Zhuhai, 519087, P.R. China} 


\begin{abstract}
We discretize the Vlasov-Poisson system using 
conservative semi-Lagrangian (CSL) discontinuous Galerkin (DG) schemes
that are asymptotic preserving (AP) in the quasi-neutral limit.
The proposed method (CSLDG) relies on two key ingredients: the CSLDG discretization and a reformulated Poisson equation (RPE).
The use of the CSL formulation ensures local mass conservation and circumvents the Courant-Friedrichs-Lewy condition, while the DG method provides high-order accuracy for capturing fine-scale phase space structures of the distribution function.
The RPE is derived by the Poisson equation coupled with
moments of the Vlasov equation.
The synergy between the CSLDG and RPE components makes it possible to obtain reliable numerical solutions, even when the spatial and temporal resolution might not fully resolve the Debye length.
We rigorously prove that the proposed method is asymptotically stable, consistent and satisfies AP properties. 
Moreover, its efficiency is maintained across non-quasi-neutral and quasi-neutral regimes. 
These properties of our approach are essential for accurate and robust numerical simulation of complex electrostatic plasmas.
Several numerical experiments verify the accuracy, stability and efficiency of the proposed CSLDG schemes.
\end{abstract}



\section{Introduction}

\hspace{4pt}
Plasma modeling plays a central role in applications such as nuclear fusion, semiconductor technology, environmental purification, and materials processing. Since neither theoretical nor experimental studies give sufficient insight into the inherently complex and multiscale plasma dynamics, numerical simulations are indispensable for most problems of practical interest.
The choice of an appropriate mathematical model depends on the objectives of the numerical study and the  physical scales that need to be resolved.
At the \textbf{microscopic scale}, the motion and interactions of
 particles that constitute a molecular system can be simulated using 
particle-particle / particle-mesh methods \cite{Luty01121994}, fast multipole methods \cite{GREENGARD1987325} or the random batch method  \cite{JIN2020108877}.
At the \textbf{mesoscopic scale}, the focus shifts to the evolution of distribution functions, with representative kinetic models such as the Vlasov-Poisson and Vlasov-Maxwell systems, commonly solved using particle-in-cell (PIC)\cite{Villasenor1992,LAPENTA2017349}, Eulerian\cite{Cheng2013,Cheng2014} or semi-Lagrangian methods\cite{QIU2011DG,Cai2018}. 
At the \textbf{macroscopic scale}, attention is restricted to fluid quantities like density, velocity, pressure and magnetic field. 
Fluid models, such as the two-fluid system, magnetohydrodynamics (MHD), and Euler-Poisson equations, are derived from kinetic models via moment methods and solved using finite volume methods (FVM)\cite{tang2010gas,stv2380}, discontinuous Galerkin (DG) methods \cite{CHENG2013255,li2005locally} or flux-corrected transport finite element schemes \cite{Kuzman2017}.
This work concentrates on the mesoscopic Vlasov-Poisson system.

A fundamental characteristic of plasma dynamics is its ability to shield the electric potential acting upon it, which refers to the quasi-neutral assumption that, on macroscopic scales, the net charge density in the plasma is effectively zero.
The physical model studied involves several interacting scales, among which the Debye length $\lambda$ and the electron plasma period $\omega_p^{-1}$ are particularly important.
Realistic plasma simulations must properly capture the quasi-neutrality. 
To develop multiscale simulation tools that meet this challenging requirement, the powerful framework of asymptotic-preserving (AP) methods \cite{Jin1999,Jin2022} was introduced in the literature. Examples of AP schemes that recover the correct quasi-neutral limit without requiring refined meshes or small time steps include AP-PIC\cite{DEGOND2006,DEGOND2010,Ji2023}, AP-SL\cite{DEGOND2009}, AP-DUGKS\cite{LIU2020}, and AP conservative semi-Lagrangian\cite{LIU2025} approaches, as well as Hermite reformulations of the Vlasov-Poisson system as a hyperbolic problem \cite{blaustein2025sap}. The asymptotic behavior of
such schemes was analyzed theoretically in \cite{Degond2008,Crou2016}. 

Since Vlasov-type systems are Hamiltonian, numerical schemes must preserve fundamental invariants such as mass, energy and charge. 
Considerable progress has been made in developing conservative schemes\cite{Cheng2022,LIU2023,CHEN20117018,CHEN201573,CHEN2020109228,QT2025}. 
In particular, the conservative semi-Lagrangian formulation can rigorously preserve mass\cite{QIU2011DG}. 
Moreover, this formulation circumvents the CFL stability constraint, allowing the time step to be determined only by accuracy considerations \cite{Cai2017}.
This makes the conservative semi-Lagrangian framework a powerful and popular tool for solving transport problems in the context of kinetic simulations\cite{Cai2018,Cai2019,QIU2011FD,QIU2010FV,QIU2011DG,CROUSEILLES2010}.

The high dimensionality of Vlasov-type systems arises from the distribution function being defined on a phase space.
The curse of dimensionality makes dimensional splitting an attractive strategy for reducing computational cost by decomposing the problem into a series of one-dimensional subproblems\cite{Strang,CHENG1976330,FOREST1990105,ROSSMANITH20116203}.
This allows the solution of one-dimensional problems by combining the CSL scheme with different spatial discretization techniques, such as the finite difference method (FDM)\cite{chen2021adaptive}, FVM\cite{ZHENG2022114973,ZHENG2024}, and DG methods\cite{QIU2011DG}.
Compared with FDM and FVM, DG discretizations provide enhanced flexibility, adaptivity, and compactness, and are especially well suited for complex geometries.


The main contribution of this work is the development of asymptotic-preserving conservative semi-Lagrangian discontinuous Galerkin (AP-CSLDG) schemes for the  Vlasov-Poisson system in the quasi-neutral limit.
The proposed scheme conserves the number of particles and preserves positivity of the distribution function. Stability is maintained across multiple scales without CFL or Debye-length restrictions. The high-order accuracy of the employed DG discretization is leveraged to capture fine-scale plasma dynamics efficiently.
Moreover, we perform rigorous theoretical analysis of our scheme's AP properties, including asymptotic consistency and stability.
Finally, we demonstrate the accuracy, stability and efficiency of the proposed method through several numerical experiments.
%

The paper is organized as follows: Section 2 reviews the physical background and kinetic models. Section 3 presents the formulation of the  AP-CSLDG scheme based on dimensional splitting and describes its implementation. Section 4 is devoted to the formal asymptotic analysis of our scheme. Section~5 shows the results of numerical experiments, including Landau damping, two-stream instability, and bump-on-tail instability. Finally, Section 6 concludes the paper and discusses future research directions.

\section{The reformulated Vlasov-Poisson system and its asymptotic limit}

\hspace{4pt}
In this section, we provide a comprehensive overview of the plasma models under investigation. In the asymptotic limit, the Poisson equation of the original dimensionless VP system
degenerates into an algebraic quasi-neutrality constraint. Recognizing the challenge posed by this degeneracy, we consider a reformulated VP (RVP) system, in which a reformulated Poisson equation (RPE) is derived from the Poisson equation and macroscopic moments of the Vlasov equation. We also show that the RVP system is equivalent to the original VP system if and only if the initial conditions are well prepared.

\subsection{The dimensionless VP system}
\hspace{4pt}
We consider a simple plasma model that neglects collisions of particle populations, as well as electromagnetic effects. The electrostatic field or potential due to charged particles is regarded as an external force on the particles. Positive ions are immovable and constitute a positive background charge that is uniformly distributed in space. A model that satisfies the above assumptions is the one-component (electron) VP system. The
 dimensionless form of this system is given by
\begin{align*}
  \left\{
    \begin{aligned}
      &\frac{\partial f}{\partial t}+\textbf{v}\cdot \nabla_{\textbf{x}} f+\nabla_{\textbf{x}}\phi\cdot\nabla_{\textbf{v}}f=0,\\
      &\lambda^2\Delta_{\textbf{x}} \phi = \rho-1,
    \end{aligned}
  \right.
\end{align*}
where $f(\textbf{x},\textbf{v},t)$ is the electron distribution function,
$\phi(\textbf{x},t)$ is the electric potential, and $\rho(\textbf{x},t)=\int_{\Omega_v}f\mathrm d\textbf{v}$ is the density. The involved independent variables are the phase space coordinates $(\textbf{x},\textbf{v})\in \Omega_x \times \Omega_v \subset {\mathbb{R}}^d \times {\mathbb{R}}^d~(d=1,2,3)$ and the time $t>0$. 
The dimensionless Debye length $\lambda$ is defined as 
$$
  \lambda=\frac{\lambda_D}{x_0},\quad\lambda_D=\left(\frac{\varepsilon_0 k_B T_0}{e^2 n_0}\right)^{1/2},
$$
where $x_0$ is a characteristic length scale of the problem at hand, $\varepsilon_0$ is the vacuum permittivity, $k_B$ is the Boltzmann constant, $T_0$ is the electron temperature, $e$ is the positive elementary charge, and $n_0$ is a given constant ion density.

In this paper, we focus our attention on the case $d=1$. The corresponding phase space is two-dimensional, i.e., $(x,v) \in \Omega_x \times \Omega_v \subset \mathbb{R} \times \mathbb{R}$, and the VP system can be written as
\begin{subnumcases}{\label{VP1-dim}}
  $$\partial_t f+v~\partial_xf+\partial_x\phi~\partial_vf=0,$$ \label{Vlasov1dim1}\\
  $$\lambda^2\partial_{xx} \phi = \rho-1$$\label{Poisson1dim}
\end{subnumcases}
with $\rho(x,t)=\int_{\Omega_v}f\mathrm dv$. This model is denoted by {\rm $P^{\lambda}$} for further reference.
Once the distribution function $f$ is known, the conservative variables can be obtained by calculating the moments
\begin{align}
  (\rho,\rho u, W)^T = \int_{\Omega_v} (1,v,v^2/2)^T f \mathrm dv.
\end{align}

\subsection{The quasi-neutral VP system and its reformulation}
\hspace{4pt}
In the quasi-neutral limit $\lambda \to 0$, the Poisson equation of the VP system \eqref{VP1-dim} reduces to the quasi-neutrality constraint 
$ \rho=1$.
The quasi-neutral VP system
\begin{subnumcases}{\label{VP1-dim-qls}}
    $$\partial_t f+v~\partial_xf+\partial_x\phi~\partial_vf=0,$$\label{Vlasov1dim2}\\
    $$\rho=1$$\label{qlstate}
\end{subnumcases}
is denoted by ${\rm P}^0$. In this system, the electrostatic potential becomes the Lagrange multiplier of the constraint \eqref{qlstate}. A way to recover an explicit equation for the potential $\phi$ in this case is based on the idea of reformulating the quasi-neutral VP system \eqref{VP1-dim-qls}. To derive a suitable reformulation, we first integrate the Vlasov equation \eqref{Vlasov1dim2} with respect to the velocity variable $v$. Using the quasi-neutrality constraint \eqref{qlstate} leads to the divergence-free constraint 
\begin{align}\label{0st-ql}
  \partial_x(\rho u)=0
\end{align}
for the momentum.
Taking the first moment of the Vlasov equation \eqref{Vlasov1dim2} yields
\begin{align}\label{1st}
  \partial_t(\rho u)+\partial_xS=\rho~\partial_x\phi,
\end{align}
where $S$ is the momentum tensor defined as
\begin{align}
  S = \int_{\Omega_v} v^2 f\mathrm dv.
\end{align}
The explicit equation
\begin{align*}
  \partial_x(\rho~\partial_x\phi)=\partial_{xx}S
\end{align*}
 for the potential $\phi$
can then be obtained by taking the divergence of \eqref{1st} and using \eqref{0st-ql}.
The resulting reformulation of quasi-neutral VP system, denoted by ${\rm RP}^0$, is thus given by
\begin{subnumcases}{\label{VP1-dim-ql}}
  $$\partial_t f+v~\partial_xf+\partial_x\phi~\partial_vf=0,$$\\
  $$\partial_x(\rho~\partial_x\phi)=\partial_{xx}S.$$
\end{subnumcases}
The ${\rm RP}^0$ reformulation is equivalent to the original system \eqref{VP1-dim-qls} if and only if the initial conditions are well prepared. The corresponding requirements are formulated in
the following lemma \cite{Crou2016}.
\begin{lemma}
The reformulated Vlasov-Poisson system ${\rm RP}^0$ is equivalent to the Vlasov-Poisson system ${\rm P}^0$ if and only if the initial condition is well prepared for the quasi-neutral regime. More precisely,
$$
    {\rm RP}^0 \Rightarrow {\rm P}^0
$$
if and only if 
\begin{align*}
    \left\{
        \begin{aligned}
            &\rho(x,0)=1,\\
            &\partial_x(\rho u)(x,0)=0.
        \end{aligned}
    \right.
\end{align*}
\end{lemma}
\begin{proof}
  Taking the divergence of the first order moment of the Vlasov equation, we find that $\partial_t(\partial_x (\rho u))=0$ and thus
  $$
    \partial_x(\rho u)(x,t)=\partial_x(\rho u)(x,0).
  $$
  On the other hand, the continuity equation gives 
  $$
    \rho(x,t)=\rho(x,0)+t\partial_x(\rho u)(x,0).
  $$
  This shows that ${\rm RP}^0$ implies ${\rm P}^0$ if and only if $\rho(x,0)=1$ and $\partial_x(\rho u)(x,0)=0$, which is the case if and only if the initial data is consistent with the quasi-neutral
  state.
\end{proof}
\begin{remark}
  ${\rm P}^0 \Rightarrow {\rm RP}^0$ is obviously true, since the reformulated system ${\rm RP}^0$ is derived from the original system ${\rm P}^0$.
\end{remark}


An asymptotic-preserving RVP system for the case $\lambda>0$ can be derived
similarly. Taking the zeroth moment of the Vlasov equation \eqref{Vlasov1dim1}
with respect to the velocity variable $v$, we obtain the continuity equation 
\begin{align}\label{continuity}
  \partial_t\rho+\partial_x(\rho u)=0.
\end{align}
The first moment of the Vlasov equation \eqref{Vlasov1dim1} yields the momentum equation
\begin{align}\label{current}
  \partial_t(\rho u)+\partial_xS-\rho\partial_x\phi=0.
\end{align}
Differentiating \eqref{continuity} with respect to $t$ and \eqref{current} with respect to $x$,  we eliminate the momentum $\rho u$ and arrive at
\begin{align}\label{elmJ}
  \partial_{tt}\rho-\partial_{xx}S+\partial_x(\rho\partial_x\phi)=0.
\end{align}
Invoking the Poisson equation \eqref{Poisson1dim}, we obtain the RPE
\begin{align*}
  -\partial_x[(\lambda^2\partial_{tt}+\rho)\partial_x\phi]=-\partial_{xx}S.
\end{align*}
The resulting generalized RVP system, denoted by ${\rm RP}^{\lambda}$, reads
\begin{subnumcases}{\label{RVP}}
  $$\partial_t f+v~\partial_xf+\partial_x\phi~\partial_vf=0,$$\\
  $$-\partial_x[(\lambda^2\partial_{tt}+\rho)\partial_x\phi]=-\partial_{xx}S.$$\label{RPE}
\end{subnumcases}

This reformulated system represents an appropriate model for plasma simulations
under conditions that correspond or are close to the quasi-neutral regime. As
in the case $\lambda=0$, the following result regarding the equivalence of
${\rm RP}^{\lambda}$ and ${\rm P}^{\lambda}$ can easily be proved.

\begin{lemma}
The ${\rm RVP}$ system {\rm ($\text{RP}^{\lambda}$)} implies the ${\rm VP}$ system {\rm ($\text{P}^{\lambda}$)} if and only if the electric potential and its time derivative satisfy the two Poisson equations at the initial time.
More precisely,
\begin{align*}
  {\rm RP}^{\lambda} \Rightarrow {\rm P}^{\lambda}
\end{align*}
if and only if
\begin{align}
  \left\{
  \begin{aligned}
    &\lambda^2 \partial_{xx}\phi(x,0) = \rho(x,0)-1,\\
    &\lambda^2 \partial_t\partial_{xx}\phi(x,0)=-\partial_x(\rho u)(x,0).
  \end{aligned}
  \right.
\end{align}
\end{lemma}
\begin{proof}
  Taking the divergence of the momentum equation \eqref{current}, we obtain
  $$
    \partial_t\left(\partial_x (\rho u)\right)=-\partial_{xx}S+\partial_x \left(\rho \partial_x \phi\right) .
  $$
  Using \eqref{RPE} and integrating with respect to $t$ yields
  $$
    \partial_x (\rho u)(x, t)=-\lambda^2 \partial_t \partial_{xx} \phi(x, t)+\left(\partial_x (\rho u)(x, 0)+\lambda^2 \partial_t \partial_{xx} \phi(x, 0)\right) .
  $$
  Finally, inserting this result into the continuity equation \eqref{continuity} leads to
  $$
    \rho(x, t)-1=\lambda^2 \partial_{xx} \phi(x, t)+\left(\rho(x, 0)-1-\lambda^2 \partial_{xx} \phi(x, 0)\right)+t\left(\partial_x (\rho u)(x, 0)+\lambda^2 \partial_t \partial_{xx} \phi(x, 0)\right),
  $$
  which proves the assertion of the lemma.
\end{proof}

\section{Asymptotic-preserving conservative semi-Lagrangian DG schemes}

\hspace{4pt}
To solve the RVP system numerically, we first split it into two parts: a linear system $H_{\mathbf{f}}$ and a nonlinear system $H_{\mathbf{E}}$. Next, we introduce the conservative semi-Lagrangian discontinuous Galerkin (CSLDG) method for the 1D linear transport equation to be solved in the $H_{\mathbf{f}}$ stage. The RPE for the electrostatic potential $\phi$ is discretized in space using the Fourier spectral method. A positivity-preserving limiter is applied to ensure physical admissibility of numerical solutions. Finally, we combine the solvers for the subproblems $H_{\mathbf{f}}$  and $H_{\mathbf{E}}$ of the dimension-splitting method. The resulting algorithm belongs to the family of asymptotic-preserving conservative semi-Lagrangian discontinuous Galerkin (AP-CSLDG) schemes.

\subsection{Dimension-splitting method}
\hspace{4pt}
As mentioned in the outline of this section,
we split the RVP system \eqref{RVP} into the linear system
\begin{align}
   \mbox{$H_{\mathbf{f}}$:}\quad  \left\{
    \begin{aligned}
        &\partial_tf+v\partial_xf=0,\\
        &\partial_t(\partial_x\phi)=0
    \end{aligned}
    \right.
\end{align}
and the nonlinear system 
\begin{align}
  \mbox{$H_{\mathbf{E}}$:} \qquad  \left\{
    \begin{aligned}
        &\partial_tf+\partial_x\phi\partial_vf=0,\\
        &-\partial_x[(\lambda^2\partial_{tt}+\rho)\partial_x\phi]=-\partial_{xx}S.
    \end{aligned}
    \right.
\end{align}
The exact evolution operators of the two subproblems
are denoted by $\mathcal{S}_{\mathbf{f}}$ and $\mathcal{S}_{\mathbf{E}}$, respectively.

Given a finite time $T>0$, we consider discrete time levels $0=t^0<t^1<\cdots<t^N=T$
corresponding to a uniform partition of 
$[0,T]$ into $N$ time intervals of length $\Delta t=\frac{T}{N}$. The numerical solution of \eqref{RVP} is advanced from $t^n$ to $t^{n+1}$ via the Lie splitting method consisting of two substeps:
\begin{subnumcases}{\label{Liesplitting}}
    $$f^*(x,v)=\mathcal{S}_{\mathbf{f}}(\Delta t)f(x,v,t^n),$$\\
    $$f(x,v,t^{n+1})=\mathcal{S}_{\mathbf{E}}(\Delta t)f^*(x,v)$$
\end{subnumcases}
or the Strang splitting method consisting of three substeps:
\begin{subnumcases}{\label{Strangsplitting}}
    $$f^*(x,v)=\mathcal{S}_{\mathbf{f}}(\frac{\Delta t}{2})f(x,v,t^n),$$\\
    $$f^{**}(x,v)=\mathcal{S}_{\mathbf{E}}(\Delta t)f^*(x,v),$$\\
    $$f(x,v,t^{n+1})=\mathcal{S}_{\mathbf{f}}(\frac{\Delta t}{2})f^{**}(x,v).$$
\end{subnumcases}
Under suitable smoothness assumptions, the fractional step method using the Lie splitting scheme~\eqref{Liesplitting} or the Strang splitting scheme \eqref{Strangsplitting} is first-order or second-order accurate, respectively.

In the following description of time-discrete schemes, we denote the density by $\rho^n=\rho(x,t^n)$, momentum by $(\rho u)^n=(\rho u)(x,t^n)$, momentum tensor by $S^n=S(x,t^n)$ and electrostatic potential by $\phi^n=\phi(x,t^n)$. To solve the nonlinear system $H_{\mathbf{E}}$, we need to freeze the coefficient of the flux term, i.e, the derivative of the electrostatic potential $\partial_x\phi$. The details of this linearization are provided below.

\subsubsection{Lie splitting method}

If the quasi-neutral RVP system is solved using the algorithm based on the
Lie splitting scheme \eqref{Liesplitting}, the
electrostatic potential $\phi$ should satisfy
\cite{Crou2016}
\begin{align}\label{RPEqslim}
    \partial_x \left(\rho^n \partial_x \phi\right)=\partial_{xx}S^n.
\end{align}
A first-order temporal semi-discretization of the RPE \eqref{RPE} is given by
\cite{DEGOND2010,Crou2016}
\begin{align}\label{RPEdis1st}
    -\partial_x\left[\left(\lambda^2+\rho^n \Delta t^2\right) \partial_x \phi^{n+1} \right]=-\Delta t^2 \partial_{x x} S^n+\Delta t\partial_x(\rho u)^n-\rho^n+1.
\end{align}
It is asymptotic preserving, since \eqref{RPEqslim} is the quasi-neutral limit of \eqref{RPEdis1st}.

\subsubsection{Strang splitting method}

As shown in \cite{Crou2016}, there is no easy way to construct second-order schemes preserving quasi-neutral states in the context of Strang's splitting for system \eqref{RVP}. Nevertheless, we can still freeze the coefficient $\partial_x\phi^{n+\frac{1}{2}}=\partial_x\phi\left(x,t^n+\frac{1}{2}\Delta t\right)$ as proposed in \cite{CHENG1976330}. When it comes to calculating a numerical approximation to $\partial_x\phi^{n+\frac{1}{2}}$, we discretize the RPE \eqref{RPE} in time as follows:
\begin{align}\label{RPEdis2nd} 
    -\partial_x\left[\left(\lambda^2-\frac{\Delta t^2}{24} \rho^{n+\frac{1}{2}}\right) \partial_x \phi^{n+\frac{1}{2}}\right] =\frac{\Delta t^2}{24} \partial_{x x} S^{n+\frac{1}{2}}  +\frac{\Delta t}{3} \partial_x(\rho u)^{n+\frac{1}{2}}+\frac{\Delta t}{6} \partial_x(\rho u)^n-\rho^n+1 .
\end{align}
The second-order accuracy of the above discretization is shown in the following theorem.

\begin{theorem}
The time discretization \eqref{RPEdis2nd} of the RPE \eqref{RPE} is second-order accurate. More precisely, if $\bar{\phi}^{n+\frac{1}{2}}$ is a solution of \eqref{RPEdis2nd} and $\phi^{n+\frac{1}{2}}$ is a solution of the Poisson equation \eqref{Poisson1dim}, then
    \begin{align*}
        \bar{\phi}^{n+\frac{1}{2}}-\phi^{n+\frac{1}{2}}=\mathcal{O}(\Delta t^2)
    \end{align*}
or
    \begin{align*}
        \partial_x\bar{\phi}^{n+\frac{1}{2}}-\partial_x\phi^{n+\frac{1}{2}}=\mathcal{O}(\Delta t^2).
    \end{align*}
\end{theorem}
\begin{proof}
Using Taylor expansions of $\rho^n$ and $\partial_t\rho^n$ about $t^{n+\frac{1}{2}}$,
we deduce 
\begin{align*}
    \rho^n & =\rho^{n+\frac{1}{2}}-\frac{\Delta t}{2} \partial_t \rho^{n+\frac{1}{2}}+\frac{\Delta t^2}{8} \partial_{t t} \rho^{n+\frac{1}{2}}-\frac{\Delta t^3}{48} \partial_{t t t} \rho^{n+\frac{1}{2}}+\frac{\Delta t^4}{384} \partial_{t t t t} \rho^{n+\frac{1}{2}}+\mathcal{O}\left(\Delta t^5\right), \\
    \partial_t \rho^n & =\partial_t \rho^{n+\frac{1}{2}}-\frac{\Delta t}{2} \partial_{t t} \rho^{n+\frac{1}{2}}+\frac{\Delta t^2}{8} \partial_{t t t} \rho^{n+\frac{1}{2}}-\frac{\Delta t^3}{48} \partial_{t t t t} \rho^{n+\frac{1}{2}}+\mathcal{O}\left(\Delta t^4\right).
\end{align*}
The elimination of $\partial_{ttt} \rho^{n+\frac{1}{2}}$ yields the second-order approximation
\begin{align}\label{2ndeq}
    \partial_{tt} \rho^{n+\frac{1}{2}}=\frac{-24 \rho^{n+\frac{1}{2}}+24 \rho^n+8 \Delta t \partial_t \rho^{n+\frac{1}{2}}+4 \Delta t \partial_t \rho^n}{\Delta t^2}+\frac{\Delta t^2}{192} \partial_{t t t t} \rho^{n+\frac{1}{2}}+\mathcal{O}\left(\Delta t^3\right) 
\end{align}
to equation \eqref{elmJ}. Replacing $\partial_t\rho^n$ with $-\partial_x(\rho u)^n$ in view of the continuity equation \eqref{continuity} and $\rho^{n+\frac{1}{2}}$ with $\lambda^2\partial_{xx}\phi^{n+\frac{1}{2}}+1$ in view of the Poisson equation \eqref{Poisson1dim}, we obtain
\begin{align*}
    -\partial_x\left[\left(\lambda^2-\frac{\Delta t^2}{24} \rho^{n+\frac{1}{2}}\right) \partial_x \phi^{n+\frac{1}{2}}\right] =\frac{\Delta t^2}{24} \partial_{x x} S^{n+\frac{1}{2}}  +\frac{\Delta t}{3} \partial_x(\rho u)^{n+\frac{1}{2}}+\frac{\Delta t}{6} \partial_x(\rho u)^n-\rho^n+1.
\end{align*}
The validity of the claim follows from this result.
\end{proof}

\subsection{CSLDG method}
\hspace{4pt}
Using a directional splitting procedure, the Vlasov equation is decomposed into a sequence of 1D transport problems associated with the $x$- and $v$-directions. We discretize these subproblems in space and time using the highly accurate and efficient CSLDG method \cite{QIU2011DG,Cai2017}. To introduce the semi-Lagrangian discretization procedure in a simple setting, we consider the 1D linear transport equation
\begin{align}\label{trans1D}
    \left\{
        \begin{aligned}
            &u_t+(a(x,t)u)_x=0,\quad x\in \Omega,\enspace t\in [0,T]\\
            &u(x,t=0)=u_0(x)
        \end{aligned}
    \right.
\end{align}
with periodic boundary conditions. We assume that $a(x,t)$ is continuous with respect to $x$ and $t$.

Let $\mathcal{E}_x: x_L=x_{\frac{1}{2}}<x_{\frac{3}{2}}<\cdots<x_{N_x-\frac{1}{2}}<x_{N_x+\frac{1}{2}}=x_R$ be a partition of the domain $\Omega=[x_L,x_R]$ into intervals $I_{j}:=\left[x_{j-\frac{1}{2}},x_{j+\frac{1}{2}}\right]~(j=1,2,\cdots,N_x)$ of length $\Delta x_j=x_{j+\frac{1}{2}}-x_{j-\frac{1}{2}}$. The midpoint of $I_{j}$ is denoted by $x_j=\frac{x_{j+\frac{1}{2}}+x_{j-\frac{1}{2}}}{2}$. We seek approximations in the  finite-dimensional space
$$
    V_h^k=\{v:v|_{I_j}\in P^k(I_{j})\},
$$
where $P^k(I_{j})$ is the space of polynomials of degree at most $k$ on the mesh element $I_{j}$.

The time-dependent test functions $\psi(x,t)$ of the CSLDG method proposed in \cite{Cai2017} satisfy 
\begin{align}\label{adjont}
    \left\{
        \begin{aligned}
            &\psi_t+a(x, t) \psi_x=0, \\
            &\psi\left(t=t^{n+1}\right)=\Psi,~t\in[t^n,t^{n+1}]
        \end{aligned}
    \right.
\end{align}
for any choice of $\Psi \in P^k(I_{j})$. As before,
 $t^n$ denotes the $n$th time level. We remark that the solution of the linear
1D problem \eqref{adjont} stays constant along the characteristics.
 
To derive a conservative Lagrangian weak form of \eqref{trans1D},
we use the Reynolds transport theorem
\begin{align}\label{SL}
    \frac{\mathrm d}{\mathrm dt}\int_{I_j(t)}u\psi\mathrm dx
    =\int_{I_j(t)}(u\psi)_t+(a(x,t)u\psi)_x\mathrm dx=0,
\end{align}
where $I_j(t)$ is the time-dependent interval delimited by the characteristics emanating from the endpoints $x_{j\pm\frac{1}{2}}$ of $I_j$ at $t=t^{n+1}$,
as shown in Fig. \ref{fig:trajectory1}. The endpoints of $
I_j^*:=I_j(t^n)$ are denoted by $x_{j\pm\frac{1}{2}}^*$.

The use of \eqref{SL} leads to a semi-Lagrangian DG method for calculating
$u^{n+1}\in V_h^k$ such that 
\begin{align}\label{SLscheme}
    \int_{I_j} u^{n+1} \Psi\mathrm d x=\int_{I_j^*} u^n \psi^n\mathrm d x
\qquad\forall \Psi  \in V_h^k.
\end{align}
\begin{figure}[h!]
    \centering    
    \begin{subfigure}[t]{0.48\linewidth}
        \centering
        \includegraphics[width=\linewidth]{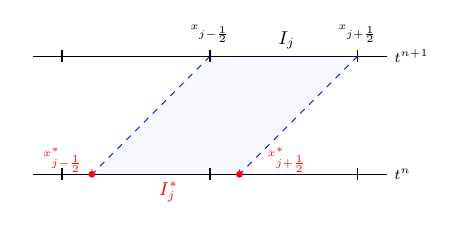}
        \caption{Characteristics and the upstream cell.}
        \label{fig:trajectory1}
    \end{subfigure}
    \hfill
    \begin{subfigure}[t]{0.48\linewidth}
        \centering
        \includegraphics[width=\linewidth]{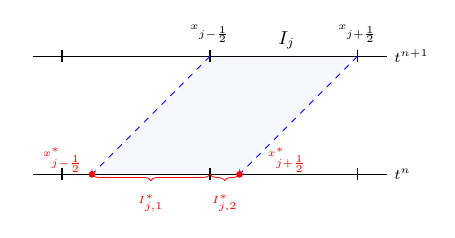}
        \caption{Subdivision of the upstream cell.}
        \label{fig:trajectory2}
    \end{subfigure}
    \caption{Schematic illustration of the CSLDG scheme in 1D.}
    \label{fig:trajectory}
\end{figure}
To advance the numerical solution in time from $t^n$ to 
$t^{n+1}$, we need to execute the following steps:
\begin{enumerate}
    \item Construct characteristics through the quadrature points $\{x_{j,q}\}_{q=1,2,\cdots,k+1}$ of a numerical integration rule for $I_j$ and
  locate their feet $\{x_{j,q}^*\}_{q=1,2,\cdots,k+1}$
  by solving the initial value problems
    $$
        \frac{\mathrm dx}{\mathrm dt}=a(x,t),\quad x(t^{n+1})=x_{j,q},\quad t\in[t^n,t^{n+1}].
    $$
    \item
    Approximate the restriction of $\psi(x,t^n)$ to the interval $I_j^*$
    by a polynomial $\psi^*(x)$ of degree $k$ fitted to
    the data $\{\left(x_{j,q}^*,\Psi(x_{j,q})\right)\}_{q=1,2,\cdots,k+1}$, i.e.,
    such that $\psi^*(x_{j,q}^*)=\Psi(x_{j,q})$.
    \item Detect intervals or subintervals $I_{j,l}^*$ that represent intersections 
    of $I_j^*=\cup_lI_j^*$ with elements $I_k$ of the fixed mesh (as shown in Fig.
    \ref{fig:trajectory2} for $l=1,2$ and $k=j-1,j$).
    \item Perform numerical integration to calculate
    (an approximation to) the right-hand side
    \begin{align}\label{rhsSL}
       \sum_l\int_{I_{j,l}^*}u^n\psi^*\mathrm dx\approx 
       \int_{I_j^*}u^n\psi^n\mathrm dx.
    \end{align}

    \item Use the basis functions of the polynomial space $P^k(I_{j})$ as test
    functions $\Psi$ in the discrete version of \eqref{SLscheme}
    and solve the resulting linear system for the local degrees of freedom.
    
    \item Repeat the above procedure for all elements $I_j$.
\end{enumerate}
\begin{lemma}
    The CSLDG scheme \eqref{SLscheme} is locally mass conservative in the sense that
    $$
        \int_{I_j}u^{n+1}\mathrm dx=\int_{I_j^*}u^n\mathrm dx.
    $$
    In the case of periodic boundary conditions, the global mass is also
    conserved, i.e.,
   $$
        \int_{\Omega}u^{n+1}\mathrm dx=\int_{\Omega}u^n\mathrm dx.
    $$
\end{lemma}
\begin{proof}
Using $\Psi=1$ in \eqref{SLscheme}, we find that the statement regrading local
mass conservation is true. Summation over all elements $I_j$ proves the
global conservation property in the periodic case.
\end{proof}

\subsection{Calculation of the electrostatic potential}\label{sec:phi}
\hspace{4pt}
To linearize the subproblem $H_{\mathbf{E}}$ of the dimension-splitting method, we need to freeze the electrostatic potential $\phi$ corresponding to a numerical solution of the semi-discrete RPE.
Since $\phi(x,t)$ is periodic in the physical space, the Fourier spectral method is a natural choice for the spatial discretization of \eqref{RPEdis1st} or \eqref{RPEdis2nd}. This way to obtain $\phi^{n+1}$ or $\phi^{n+\frac{1}{2}}$  not only provides high accuracy in space but can also significantly reduce the computational time by employing the fast Fourier transformation (FFT).

For numerical integration purposes, the frozen coefficients $\partial_x\phi^{n+1}$ or $\partial_x\phi^{n+\frac{1}{2}}$ need to be evaluated at the Gaussian quadrature points $x_{j,q}$ of the element $I_j$ \cite{QIU2011DG}. However, the approximations produced by the Fourier spectral method for the semi-discretization of RPE \eqref{RPEdis1st} or \eqref{RPEdis2nd} are defined only at the cell midpoints $\{x_j\}$ of the uniform mesh. Therefore, it is necessary to employ certain reconstruction techniques to recover the values at the Gaussian points. 

The procedure for obtaining $\partial_x\phi^{n+1}$ for the Lie splitting method is as follows. We first evaluate the density $\rho^n$, momentum $(\rho u)^n$ and momentum tensor $S^n$ at $\{x_j\}$ by taking moments of the numerical solution $f_h^n$ to the subproblem $H_{\mathbf{E}}$. Then, we use the Fourier spectral discretization of the RPE and the FFT algorithm to calculate the values of $\partial_x\phi^{n+1}$ at $\{x_j\}$. Finally, we use interpolation polynomials of degree $k$ to reconstruct $\partial_x\phi^{n+1}(x)$ and calculate the values at $\{x_{j,q}\}$.

A similar algorithm is used to obtain $\partial_x\phi^{n+\frac{1}{2}}$ for the Strang splitting method. The only difference is that we need to evaluate the density $\rho^{n+\frac{1}{2}}$, current density $J^{n+\frac{1}{2}}$ and momentum tensor $S^{n+\frac{1}{2}}$ at $\{x_j\}$ using the numerical solution $f^{*}$ of the linear subproblem $H_{\mathbf{f}}$ after the first halved time step.


\subsection{Positivity-preserving limiter}
\hspace{4pt}
Since the distribution function $f(x,v,t)$ that represents the exact
solution of the VP subproblem $H_{\mathbf{f}}$ is non-negative, we should
ensure positivity preservation for our numerical solution $f_h^n(x,v)$. 
A simple scaling limiter that guarantees preservation of local and/or global
bounds in the context of Runge-Kutta discontinuous Galerkin (RKDG)
methods for hyperbolic problems was presented in \cite{PP-36,PP-37}.
Under the inductive assumption that the cell averages $\bar{f}_j=\frac{1}{\Delta x_j}\int_{I_j} f_h(x)\mathbf dx$ are maximum principle preserving (MPP)
and/or positivity preserving (PP), the corresponding constraints for
pointwise values are enforced by limiting the deviation of $f_h|_{I_j}$
from the average value
$\bar{f}_j$. In essence, the numerical solution $f_h(x)$ is
replaced by a slope-limited approximation  $\tilde{f}_h(x)$ that
satisfies the following conditions:
\begin{enumerate}
    \item Accuracy: For smooth solutions, $\Vert f_h-\tilde{f}_h\Vert_{\infty}=\mathcal{O}(\Delta x^{k+1})$, where $k$ is the degree of the polynomial approximation used in the DG method.
    \item Conservation: $\int_{I_j} \tilde{f}_h(x)\mathrm dx=\int_{I_j} f_h(x)\mathrm dx$.
    \item Positivity: $\tilde{f}_h(x)\geq 0$ for all $x\in \Omega_x$.
\end{enumerate}
In our asymptotic-preserving conservative semi-Lagrangian discontinuous Galerkin (AP-CSLDG) scheme, we apply the PP limiter to $f_h(x)$ after each time step:
\begin{align*}
    \tilde{f}_h(x)=\theta_j\left(f_h(x)-\bar{f}_j\right)+\bar{f},\quad \theta_j=\min\left\{\left\vert \frac{m_0-\bar{f}}{m_j'-\bar{f}}\right\vert,1\right\},
    \qquad x\in I_j,
\end{align*}
where $m_0=10^{-15}$ and $m_j'$ represents the minimum value of $f_h(x)$ within the element $I_j$. This limiting strategy has been rigorously proven to preserve positivity
\cite{PP-35,PP-37} without losing the local conservation property or degrading
the accuracy of a high-order DG approximation in smooth regions.

\subsection{Fractional step algorithms}\label{sec:algorithm}
\hspace{4pt}
In this subsection, we describe our implementation of
the fractional step algorithm 
corresponding to the Lie splitting
version of our AP-CSLDG scheme for the RVP system.
The implementation
of the Strang splitting version is similar.
Additional details and explanations can be found in \cite{QIU2011DG}.
The asymptotic-preserving properties of our scheme are analyzed in the next section.

\subsubsection{Tensor-product DG discretization}
\hspace{4pt}
Consider the 1D1V computational domain $\Omega=\Omega_x\times\Omega_v$ and uniform partitions $\mathcal{E}_x$ and $\mathcal{E}_v$ of the domains $\Omega_x=[x_L,x_R]$ and $\Omega_v=[v_{\text{min}},v_{\text{max}}]$ into $N_x$ and $N_v$ cells, respectively. The two-dimensional phase domain $\Omega$ is discretized using the tensor product $\mathcal{E}=\mathcal{E}_x\otimes \mathcal{E}_v$ of the uniform 1D meshes $\mathcal{E}_x=\{I_{x_j}\}_{j=1}^{N_x}$ and $ \mathcal{E}_v=\{I_{v_i}\}_{i=1}^{N_v}$ for the physical domain $\Omega_x$ and the velocity domain 
$\Omega_v$. The so-defined regular partition $\mathcal{E}$ of 
$\Omega$ consists of rectangular cells $I_{x_j}\otimes I_{v_i}$. The
discontinuous
finite element approximation space $V_h^k(\mathcal{E})=\mathcal{D}_k(\mathcal{E}_x)\otimes \mathcal{D}_k(\mathcal{E}_v)$ is defined as the tensor product of the
spaces
\begin{align*}
    &\mathcal{D}_k(\mathcal{E}_x) = \{f_h:f_h|_{I_{x_j}}\in P^k(I_{x_j}),~j=1,2,\cdots,N_x\},\\
    &\mathcal{D}_k(\mathcal{E}_v) = \{f_h:f_h|_{I_{v_i}}\in P^k(I_{v_i}),~i=1,2,\cdots,N_v\}. 
\end{align*}
Therefore, there are $(k+1)^2$ degrees of freedom in each element.

\subsubsection{Numerical approximation of the operator $\mathcal{S}_{\mathbf{f}}$}
\hspace{4pt}
To solve the linear problem $H_{\mathbf{f}}$, we first fix $k+1$ Gaussian points
 $\{v_{i,p}\}_{p=1,2,\cdots,k+1}$ in each element $I_{v_i}$ of the 1D mesh $\mathcal E_v$. 
Then for each discrete velocity $v_{i,p}$, we solve the $x$-direction transport equation
\begin{align*}
    \partial_tf+v_{i,p}\partial_xf=0
\end{align*}
using the conservative SLDG scheme \eqref{SLscheme} to
find $f_h^{*}(x,v_{i,p})\in \mathcal{D}_k(\mathcal{E}_x)$ such that
\begin{align}\label{Hhf}
    \int_{I_{x_j}} f_h^{*}(x,v_{i,p}) \Psi(x)\mathrm d x
    =\int_{I_{x_j}^*} f_h^n(x,v_{i,p}) \psi^n(x)\mathrm d x
\end{align}
for any $\Psi(x) \in \mathcal{D}_k(\mathcal{E}_x)$. The interval
$I_{x_j}^*=\left[x_{j-\frac{1}{2}}^*,x_{j+\frac{1}{2}}^*\right]$ is
delimited by the feet $x_{j\pm \frac{1}{2}}^*=x(t^n)$ of the two
characteristics $x(t)$ that are constructed by solving the initial value problems
\begin{align*}
    \frac{\mathrm dx}{\mathrm dt}=v_{i,p},\quad x(t^{n+1})=x_{j\pm \frac{1}{2}},\quad t\in[t^n,t^{n+1}].
\end{align*}
At the end of the above CSLDG update, we enforce positivity preservation
for the values of $f_h^{*}(x,v)$ at the Gaussian points $(x_{j,q},v_{i,p})$
of each rectangular element $I_{x_j}\times I_{v_i}$ using the PP scaling
limiter. The slope-limited approximation to $f_h^{*}(x,v)$
provides an admissible initial condition for the next step.

\subsubsection{Numerical approximation of the operator $\mathcal{S}_{\mathbf{E}}$}
\hspace{4pt}
To solve the nonlinear problem $H_{\mathbf{E}}$, we first calculate the frozen coefficients $\partial_x\phi^{n+1}$ at the Gaussian points $\{x_{j,q}\}$ using the procedure introduced in Section \ref{sec:phi}. 
Then for each discrete physical location $x_{j,q}$, we solve the $v$-direction transport equation
\begin{align*}
    \partial_tf+\partial_x\phi^{n+1}(x_{j,q})\partial_vf=0
\end{align*}
using the conservative SLDG scheme \eqref{SLscheme} to find $f_h^{n+1}(x_{j,q},v)\in \mathcal{D}_k(\mathcal{E}_v)$ such that
\begin{align}\label{HhE}
    \int_{I_{v_i}} f_h^{n+1}(x_{j,q},v) \Psi(v)\mathrm d v=\int_{I_{v_i}^*} f_h^{*}(x_{j,q},v) \psi^n(v)\mathrm d v
\end{align}
for any $\Psi(v) \in \mathcal{D}_k(\mathcal{E}_v)$. The interval $I_{v_i}^*=\left[v_{i-\frac{1}{2}}^*,v_{i+\frac{1}{2}}^*\right]$ is delimited by the feet $v_{i\pm \frac{1}{2}}^*=v(t^n)$ of the two characteristics $v(t)$ that are constructed by solving the initial value problems
\begin{align*}
    \frac{\mathrm dv}{\mathrm dt}=\partial_x\phi^{n+1}(x_{j,q}),\quad v(t^{n+1})=v_{i\pm \frac{1}{2}},\quad t\in[t^n,t^{n+1}].
\end{align*}
The application of the PP scaling limiter to the updated functions
$f_h^{n+1}(x,v)$ ensures positivity preservation at the Gaussian
points $(x_{j,q},v_{i,p})$ of each rectangular element $I_{x_j}\times I_{v_i}$. 

\subsubsection{Structure of the Lie and Strang splitting schemes}
\hspace{4pt}
Let $\mathcal{S}^{h}_{\mathbf{f}}$ and $\mathcal{S}^{h}_{\mathbf{E}}$ denote the numerical counterparts of the evolution operators $\mathcal{S}_{\mathbf{f}}$ and $\mathcal{S}_{\mathbf{E}}$, respectively. The AP-CSLDG scheme based on the Lie splitting method consists of the two steps
\begin{align*}
    &f_h^{*}(x,v)=\mathcal{S}^{h}_{\mathbf{f}}(\Delta t)f_h^n(x,v),\\
    &f_h^{n+1}(x,v)=\mathcal{S}^{h}_{\mathbf{E}}(\Delta t)f_h^{*}(x,v).
\end{align*}
This two-step algorithm is referred to as AP-CSLDG-1 in what follows.

The AP-CSLDG scheme based on the Strang splitting method
consists of the three steps
\begin{align*}
    &f_h^{*}(x,v)=\mathcal{S}^{h}_{\mathbf{f}}(\frac{\Delta t}{2})f_h^n(x,v),\\
    &f_h^{**}(x,v)=\mathcal{S}^{h}_{\mathbf{E}}(\Delta t)f_h^{*}(x,v),\\
    &f_h^{n+1}(x,v)=\mathcal{S}^{h}_{\mathbf{f}}(\frac{\Delta t}{2})f_h^{**}(x,v)
\end{align*}
and is denoted by AP-CSLDG-2.

It is easy to verify that our AP-CSLDG schemes for the RVP system
preserve the total number of particles
$\int_{\Omega_x}\rho(x)\mathrm dx=
\int_{\Omega_v}\int_{\Omega_x} f_h(x)\mathrm d x\mathrm d v$.
In the following theorem, we show that this is indeed the case
for the AP-CSLDG-1 version. The proof for the AP-CSLDG-2 scheme is similar.
\begin{theorem}[Conservation of the particle load] The AP-CSLDG-1 scheme preserves the global mass
    \begin{align*}
        \int_{\Omega_v}\int_{\Omega_x} f_h^{n+1}(x,v)\mathrm d x\mathrm d v=\int_{\Omega_v}\int_{\Omega_x} f_h^n(x,v)\mathrm d x\mathrm d v.
    \end{align*}
\end{theorem}
\begin{proof}
   The numerical operator $\mathcal{S}_{\mathbf{f}}^h$
   of the SLDG scheme \eqref{Hhf} has the mass conservation property
    \begin{align*}
        \int_{\Omega_x} f_h^{*}(x,v_{i,p})\mathrm d x=\int_{\Omega_x} f_h^n(x,v_{i,p})\mathrm d x.
    \end{align*}
Using the Gaussian quadrature with $k+1$ points for integration
over $\Omega_v$, we obtain
    \begin{align*}
        \int_{\Omega_v}\int_{\Omega_x} f_h^{*}(x,v)\mathrm d x\mathrm d v &=\sum_{i=1}^{N_v}\sum_{p=1}^{k+1} w_{i,p} \int_{\Omega_x} f_h^{*}(x,v_{i,p})\mathrm d x\\
        &=\sum_{i=1}^{N_v}\sum_{p=1}^{k+1} w_{i,p} \int_{\Omega_x} f_h^{n}(x,v_{i,p})\mathrm d x
        =\int_{\Omega_v}\int_{\Omega_x} f_h^{n}(x,v)\mathrm d x\mathrm d v.
    \end{align*}
    Similarly, the design of
    the numerical operator $\mathcal{S}_{\mathbf{E}}^h$ for
    the SLDG scheme \eqref{HhE} ensures that
    \begin{align*}
        \int_{\Omega_x}\int_{\Omega_v} f_h^{n+1}(x,v)\mathrm d v\mathrm d x=\int_{\Omega_x}\int_{\Omega_v} f_h^{*}(x,v)\mathrm d v\mathrm d x.
    \end{align*}
    Combining the above integral conservation laws for $\mathcal{S}_{\mathbf{f}}^h$
    and $\mathcal{S}_{\mathbf{E}}^h$
    proves the validity of the claim.
\end{proof}

\section{Analysis of the AP-CSLDG schemes}

\hspace{4pt}
A good numerical method for plasma flow simulations should provide
continuous dependence on the dimensionless Debye length $\lambda$ and ensure
correct asymptotic behavior in the quasi-neutral regime, in which $\lambda\to 0$.  
This requirement motivated the development of asymptotic-preserving (AP) numerical
schemes based on PIC \cite{DEGOND2006, DEGOND2010}, SL \cite{DEGOND2009}, DUGKS \cite{LIU2020} and CSL \cite{LIU2025} approaches.

In this section, we recall the definitions
of strong and weak AP properties for the RVP system \cite{Crou2016}. Introducing
some additional notation, we perform theoretical analysis that proves the
asymptotic consistency and stability of the proposed AP-CSLDG schemes
in the quasi-neutral limit.


\subsection{Notation and norms}

Adopting the notation of Section 2, we denote the  RVP system \eqref{RVP} by $\text{RP}^{\lambda}$ and its  quasi-neutral limit \eqref{VP1-dim-ql} by $\text{RP}^{0}$. The discrete 
problems resulting from the AP-CSLDG discretization \eqref{sec:algorithm} of these systems are denoted by
$\text{RP}^{\lambda}_{\delta}$ and $\text{RP}^0_{\delta}$, respectively. The resolution
parameter $\delta=\delta (\Delta x,\Delta v,\Delta t)$ indicates the dependence on the
mesh size $\Delta x$, velocity increment $\Delta v$ and time step $\Delta t$.

The standard $L^2$-norm is defined by
$$
    \Vert f \Vert = \Vert f \Vert_{L^2(\Omega)}
    =\left(\int_{\Omega_x}\int_{\Omega_v}\vert f(x,v)\vert ^2\mathrm dv\mathrm dx\right)^{\frac{1}{2}}
$$
for (finite element approximations to) distribution functions $f(x,v)$ and by
$$
    \Vert \rho \Vert = \Vert \rho \Vert_{L^2(\Omega_x)}=\left(\int_{\Omega_x}\vert \rho(x) \vert ^2\mathrm dx\right)^{\frac{1}{2}}
$$
for functions $\rho(x)$ of the space variable $x$, such as moments of $f(x,v)$.

\subsection{Definition of the AP properties}
The first AP schemes for kinetic plasma models were proposed by Jin \cite{Jin1999}. Further development of such schemes for various applications has been actively pursued in the literature \cite{DEGOND2006,DEGOND2009,DEGOND2010,Crou2016,LIU2025}. A scheme is commonly called AP if it is asymptotically stable and consistent in a limit that changes the structure of the original system. In the context of the RVP system, the following AP properties are relevant \cite{Crou2016}. 
\begin{definition}[Strong AP property]
A consistent and stable discretization {\rm $\text{RP}_\delta^{\lambda}$}
of the system {\rm $\text{RP}^{\lambda}$} is
{\rm asymptotic preserving (AP)} in the quasi-neutral limit if it is stable uniformly with respect to $\lambda$ and becomes a consistent discretization of the reduced system {\rm $\text{RP}^{0}$} in the limit $\lambda \rightarrow0$, at least for well prepared initial data. The AP property is said to be strong if it is guaranteed for any initial condition.
\end{definition}
\begin{definition}[Weak AP property] A numerical scheme that is guaranteed to be AP only for initial data consistent with the limit system is said to possess the weak AP property.
\end{definition}
The above AP properties are illustrated and summarized in Fig. \ref{fig:APproperty}.
\begin{figure}[h!]
    \centering
    \includegraphics[width=0.5\textwidth]{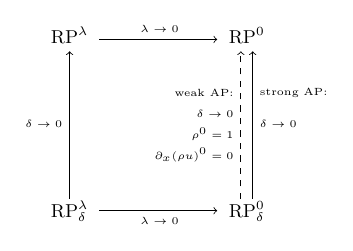} 
    \caption{Illustration of the AP properties.}
    \label{fig:APproperty}
\end{figure}
\subsection{Consistency of the AP-CSLDG schemes for the quasi-neutral RVP system}
To verify the compatibility of our AP-CSLDG schemes with the quasi-neutral limit \eqref{VP1-dim-ql} of the RVP system, we need to analyze the consistency error of the
fractional step time discretization.

\begin{theorem}\label{theorem:consistency}
Let the distribution function $f(x,v,t)$ be discretized w.r.t. the time variable $t$, while keeping the phase space variables $x$ and $v$ continuous. Evolve $f^n(x,v)$ using the Lie splitting 
\begin{align}\label{SLsemischeme}
    &f^*(x,v)=f^n(x-v\Delta t,v),\notag\\
    &f^{n+1}(x,v)=f^*(x,v-\partial_x\phi\Delta t).
\end{align}
Assume that the electric potential $\phi$ 
satisfies the quasi-neutral equation
\begin{align*}
    \partial_x(\rho^*\partial_x\phi)=\partial_{xx}S^n
\end{align*}
and that the initial conditions are consistent with the quasi-neutral limit, i.e., $\rho^0=1$ and $\partial_x(\rho u)^0=0$.

Then, for any finite time $T=N\Delta t$, the consistency error of the
time discretization satisfies
\begin{align*}
    &\Vert \rho^N - 1\Vert \leq (C_1+C_2)~\Delta t,\\
    &\Vert \partial_x(\rho u)^N\Vert \leq C_2~\Delta t,
\end{align*}
where $C_1$ and $C_2$ are independent of $\Delta t$.
\end{theorem}
\begin{proof}
Performing Taylor expansions for the two steps 
of the Lie splitting scheme \eqref{SLsemischeme}, we obtain
\begin{align*}
        f^*(x,v)&=f^n(x-v~\Delta t,v)=f^n(x,v)-v\Delta t\partial_xf^n(x,v)+\frac{v^2\Delta t^2}{2}\partial_{xx}f^n(\xi,v),\\
        f^{n+1}(x,v)&=f^*(x,v-\partial_x\phi~\Delta t)=f^*(x,v)-\Delta t\partial_x\phi~\partial_vf^*(x,v)+\frac{(\partial_x\phi)^2\Delta t^2}{2}\partial_{vv}f^*(x,\eta).
\end{align*}
Without loss of generality, we assume that $\xi\in(x-v\Delta t,x)$ and $\eta\in(v-\partial_x\phi~\Delta t,v)$. Integration over the velocity space $\Omega_v$ yields the densities
\begin{align*}
    \rho^*(x)&=\rho^n(x)-\Delta t\partial_x(\rho u)^n(x)+\frac{\Delta t^2}{2}\int_{\Omega_v}v^2\partial_{xx}f^n(\xi,v)\mathrm dv,\\
    \rho^{n+1}(x)&=\rho^*(x)+\frac{\Delta t^2}{2}(\partial_x\phi)^2\int_{\Omega_v}\partial_{vv}f^*(x,\eta)\mathrm dv.
\end{align*}
Next, we multiply the Taylor expansions of
$f^*(x,v)$ and $f^{n+1}(x,v)$ by $v$, integrate
over the velocity space $\Omega_v$ and differentiate w.r.t. $x$. These
manipulations yield the momentum gradients
\begin{align*}
    \partial_x(\rho u)^*(x)&=\partial_x(\rho u)^n(x)-\Delta t\partial_{xx}S^n(x)+\frac{\Delta t^2}{2}\partial_x\int_{\Omega_v}v^3\partial_{xx}f^n(\xi,v)\mathrm dv,\\
    \partial_x(\rho u)^{n+1}(x)&=\partial_x(\rho u)^*(x)+\Delta t\partial_x\left(\rho^*(x)\partial_x\phi\right)+\frac{\Delta t^2}{2}\partial_x\left[(\partial_x\phi)^2\int_{\Omega_v}v\partial_{vv}f^*(x,\eta)\mathrm dv\right].
\end{align*}
Invoking the quasi-neutral equation, we write the gradient of $(\rho u)^{n+1}(x)$
in the form
\begin{align*}
     \partial_x(\rho u)^{n+1}=\partial_x(\rho u)^n+\frac{\Delta t^2}{2}\partial_x\int_{\Omega_v}v^3\partial_{xx}f^n(\xi,v)\mathrm dv+\frac{\Delta t^2}{2}\partial_x\left[(\partial_x\phi)^2\int_{\Omega_v}v\partial_{vv}f^*(x,\eta)\mathrm dv\right].
\end{align*}
It follows that
\begin{align*}
    \Vert \partial_x(\rho u)^{n+1}\Vert\leq\Vert\partial_x(\rho u)^n\Vert+c_n\Delta t^2,
\end{align*}
where $c_n$ is independent of $\Delta t$. Using the initial condition $\partial_x(\rho u)^0=0$, we easily obtain
\begin{align*}
    \Vert \partial_x(\rho u)^N\Vert&\leq\Vert\partial_x(\rho u)^{N-1}\Vert+c_{N-1}\Delta t^2\\
    &\leq\Vert\partial_x(\rho u)^{N-2}\Vert+(c_{N-2}+c_{N-1})\Delta t^2\\
    &\leq\cdots\\
    &\leq\Vert\partial_x(\rho u)^0\Vert+(c_0+\cdots+c_{N-1})\Delta t^2\\
    &\leq C_2\Delta t,
\end{align*}
where $C_2=\max\{c_0,\cdots,c_{N-1}\}$. We have
\begin{align*}
    \rho^{n+1}=\rho^n-\Delta t\partial_x(\rho u)^n+\frac{\Delta t^2}{2}\int_{\Omega_v}v^2\partial_{xx}f^n(\xi,v)\mathrm dv+\frac{\Delta t^2}{2}(\partial_x\phi)^2\int_{\Omega_v}\partial_{vv}f^*(x,\eta)\mathrm dv.
\end{align*}
Therefore,
\begin{align*}
    \Vert\rho^{n+1}-\rho^n\Vert\leq\Delta t\Vert\partial_x(\rho u)^n\Vert+\bar{c}_{n}\Delta t^2,
\end{align*}
where $\bar{c}_n$ is independent of $\Delta t$. Finally, we use
the initial condition $\rho^0=1$ to show that
\begin{align*}
    \Vert \rho^N-1\Vert\leq& \Vert\rho^N-\rho^{N-1}\Vert+\Vert \rho^{N-1}-\rho^{N-2}\Vert+\cdots+\Vert\rho^1-\rho^0\Vert\\
    \leq& \Delta t\left(\Vert\partial_x(\rho u)^{N-1}\Vert+\Vert\partial_x(\rho u)^{N-2}\Vert+\cdots+\Vert\partial_x(\rho u)^0\Vert\right)\\
    &+\left(\bar{c}_{N-1}+\bar{c}_{N-2}+\cdots+\bar{c}_0\right)\Delta t^2\\
    \leq& \Delta t\left(N\Vert\partial_x(\rho u)^0\Vert+(c_{N-1}+2c_{N-2}+\cdots+Nc_0)\Delta t^2\right)\\
    &+\left(\bar{c}_{N-1}+\bar{c}_{N-2}+\cdots+\bar{c}_0\right)\Delta t^2\\
    \leq& (1+2+\cdots+N)C_2\Delta t^3+\left(\bar{c}_{N-1}+\bar{c}_{N-2}+\cdots+\bar{c}_0\right)\Delta t^2\\
    \leq& (C_1+C_2)\Delta t,
\end{align*}
where $C_1=\text{max}\{\bar{c}_0,\cdots,\bar{c}_{N-1}\}$.
\end{proof}
The above theorem shows that the Lie splitting version of the proposed SL scheme
provides a consistent first-order accurate approximation to the quasi-neutral state 
$\text{RP}^0$ of the system $\text{RP}^\lambda$ if the density and momentum of the
initial distribution are compatible with $\text{RP}^0$. This theoretical
result is significant, because it demonstrates the scheme's ability to properly capture the quasi-neutral behavior.
\begin{remark}
No matter how the electrostatic potential $\phi$ is chosen, the Strang splitting
version of the SL scheme does not preserve consistency with the quasi-neutral state.
If we start with $\phi$ such that $\partial_x(\rho^n\partial_x\phi)=\partial_{xx}S^n$, we may have $\partial_{xx}(\rho u)^{**}=0$, and then $\rho^{n+1}=1$. However,
no choice of initial data ensures that $\partial_x(\rho u)^{n+1}=0$.
    This implies that at the next time step, the scheme may not maintain the quasi-neutral state, leading to a deviation from the expected asymptotic behavior.
    Thus, it appears that there is no easy way to construct high-order
    splitting schemes preserving the quasi-neutral state \eqref{VP1-dim-ql} of the RVP system. For more details, we refer the interested reader to\cite{Crou2016}.
\end{remark}

\subsection{Stability of the AP-CSLDG-1 scheme}

The above analysis provides sufficient conditions for consistency
of the semi-discrete AP-CSLDG-1 scheme with the quasi-neutral RVP system. 
We now turn to the stability analysis for the numerical evolution operators $\mathcal{S}_{\mathbf{f}}^h$
and $\mathcal{S}_{\mathbf{E}}^h$ of the fully discrete Lie splitting scheme
\begin{align*}
    &f_h^*=\mathcal{S}_{\mathbf{f}}^h(\Delta t)f_h^n,\\
    &f_h^{n+1}=\mathcal{S}_{\mathbf{E}}^h(\Delta t)f_h^*.
\end{align*}
In view of the consistency result established in Theorem \ref{theorem:consistency},
guaranteed stability would imply that the AP-CSLDG-1 scheme is indeed asymptotically
preserving.

\begin{theorem}[$L^2$ stability of the operator $\mathcal{S}_{\mathbf{f}}^h$]
Let $f_h^*=\mathcal{S}_{\mathbf{f}}^h(\Delta t)f_h^n$. Then 
\begin{align*}
    \Vert f_h^{*}\Vert\leq \Vert f_h^n\Vert
\end{align*}
under the assumption of periodic boundary conditions.
\end{theorem}
\begin{proof}
Using $\Psi(x)=f_h^{*}(x,v_{i,p})$ in \eqref{Hhf}, we find that
\begin{align*}
    \int_{I_{x_j}}\vert f_h^{*}(x,v_{i,p})\vert^2\mathrm dx&=\int_{I_{x_j}^*}f_h^n(x,v_{i,p})f_h^{*}(x+v_{i,p}\Delta t,v_{i,p})\mathrm dx\\
    & \leq \frac{1}{2}\left(\int_{I_{x_j}}\vert f_h^*(x,v_{i,p})\vert^2\mathrm dx+ \int_{I_{x_j}^*}\vert f_h^n(x,v_{i,p})\vert^2\mathrm dx \right).
\end{align*}
Summing the above inequalities over $j=1,\ldots,N_x$ and using the assumption
of periodicity, we deduce
\begin{align*}
    \int_{\Omega_x}\vert f_h^{*}(x,v_{i,p})\vert^2\mathrm dx\leq \int_{\Omega_x}\vert f_h^n(x,v_{i,p})\vert^2\mathrm dx.
\end{align*}
The $L^2$ stability property
\begin{align*}
    \int_{\Omega_v}\int_{\Omega_x}\vert f_h^{*}(x,v_{i,p})\vert^2\mathrm dx\mathrm dv\leq \int_{\Omega_v}\int_{\Omega_x}\vert f_h^n(x,v_{i,p})\vert^2\mathrm dx\mathrm dv
\end{align*}
can now be shown by summing over the indices of the Gaussian quadrature points $v_{i,p}$.
\end{proof}
The operator $\mathcal{S}_{\mathbf{E}}^h$ is more difficult to analyze 
than $\mathcal{S}_{\mathbf{f}}^h$, since the underlying system $H_{\mathbf E}$ is
nonlinear. In the following theorem, we use linearization techniques to prove
the von Neumann stability as in \cite{Degond2008,Crou2016}.
\begin{theorem}[Linear stability of the operator $\mathcal{S}_{\mathbf{E}}^h$]
 The update $f_h^{n+1}=\mathcal{S}_{\mathbf{E}}^h(\Delta t)f_h^*$ using
 periodic boundary conditions and the linearization
 of the system $H_{\mathbf E}$ about the steady state
\begin{align*}
    \left\{
        \begin{aligned}
            &\rho=1,\\
            &\rho u=0,\\
            &\partial_x\phi=0
        \end{aligned}
    \right.
\end{align*}
is unconditionally $L^2$ stable in the sense of von Neumann stability analysis.
\end{theorem}
\begin{proof}
Using $\Psi(v)=1$ in \eqref{HhE}, summing the resulting equations
\begin{align}
    \int_{I_{v_i}}f_h^{n+1}(x_{j,q},v)\mathrm dv=\int_{I_{v_i}^*}f_h^*(x_{j,q},v)\mathrm dv
\end{align}
 over $i=1,\ldots,N_v$ and using the assumption of periodicity, we obtain
\begin{align}\label{0m}
    \int_{\Omega_v}f_h^{n+1}(x_{j,q},v)\mathrm dv=\int_{\Omega_v}f_h^*(x_{j,q},v)\mathrm dv.
\end{align}
This integral identity can be written as
\begin{align}\label{rhoequality}
    \rho^{n+1}(x_{j,q})=\rho^*(x_{j,q}).
\end{align}
Regardless of the mesh $\mathcal{E}_x$ for the discretization of $\Omega_x$, equation \eqref{rhoequality} is valid at all Gaussian quadrature points $x_{j,q}$. Summation over the indices of these points reveals that $\rho^{n+1}(x)=\rho^*(x)$ a.e. in
$\Omega_x$.

Next, we use $\Psi(v)=v$ in \eqref{HhE} to show that
\begin{align}
    \int_{I_{v_i}}f_h^{n+1}v\mathrm dv&=\int_{I_{v_i}^*}f_h^*(v+\Delta t\partial_x\phi^{n+1})\mathrm dv\notag\\
    &=\int_{I_{v_i}^*}f_h^*v\mathrm dv+\Delta t\partial_x\phi^{n+1}\int_{I_{v_i}^*}f_h^*
    \mathrm dv.
\end{align}
Summation 
over $i=1,\ldots,N_v$ under the assumption of periodicity gives
\begin{align}\label{1m}
    \int_{\Omega_v}f_h^{n+1}v\mathrm dv=\int_{\Omega_v}f_h^*v\mathrm dv
    +\Delta t\partial_x\phi^{n+1}\int_{\Omega_v}f_h^*\mathrm dv.
\end{align}
This proves that
$\frac{(\rho u)^{n+1}-(\rho u)^*}{\Delta t}-\rho^*\partial_x\phi^{n+1}=0$
a.e. in $\Omega_x$.

Combining the above auxiliary results, we consider the nonlinear semi-discrete problem
\begin{align}\label{nonlinsysHE}
    \left\{
        \begin{aligned}
            &\rho^{n+1}-\rho^*=0,\\
            &\frac{(\rho u)^{n+1}-(\rho u)^*}{\Delta t}-\rho^*\partial_x\phi^{n+1}=0,\\
            &\lambda^2\frac{\partial_{xx}\phi^{n+1}-2\partial_{xx}\phi^n+\partial_{xx}\phi^{n-1}}{\Delta t^2}+\partial_x(\rho^*\partial_x\phi^{n+1})=\partial_{xx}S^n,
        \end{aligned}
    \right.
\end{align}
where $S=\rho u^2+p$ and $p=\rho^{\gamma}$ with $\gamma>1$ is the pressure.
Linearizing the system \eqref{nonlinsysHE} about the steady state $\rho=1,~\mathcal{J}=\rho u=0,~\partial_x\phi=0$, we arrive at
\begin{align}
    \left\{
        \begin{aligned}
            &\rho^{n+1}-\rho^*=0,\\
            &\frac{\mathcal{J}^{n+1}-\mathcal{J}^*}{\Delta t}-\partial_x\phi^{n+1}=0,\\
            &\lambda^2\frac{\partial_{xx}\phi^{n+1}-2\partial_{xx}\phi^n+\partial_{xx}\phi^{n-1}}{\Delta t^2}+\partial_{xx}\phi^{n+1}=\partial_{xx}\rho^*.
        \end{aligned}
    \right.
\end{align}

Applying the spatial Fourier transform to $\rho,~\mathcal{J}$ and $\phi$, we denote the transformed variables by $\hat\rho,~\hat{\mathcal{J}}$ and $\hat\phi$, respectively. Introducing $\hat\psi^{n+1}=(\hat\phi^{n+1}-\hat\phi^n)/\Delta t$, we write
the linear system
\begin{align}
    \left\{
        \begin{aligned}
            &\hat\rho^{n+1}-\hat\rho^*=0,\\
            &\hat {\mathcal{J}}^{n+1}-\hat{\mathcal{J}}^*-ik\Delta t\hat\phi^{n+1}=0,\\
            &\hat\phi^{n+1}-\hat\phi^n-\Delta t\hat\psi^{n+1}=0,\\
            &\lambda^2\hat\psi^{n+1}-\lambda^2\hat\psi^n+\Delta t\hat\phi^{n+1}-\Delta t\hat\rho^*=0
        \end{aligned}
    \right.
\end{align}
in the matrix form
\begin{align}
B
    \left(
    \begin{array}{cccc}
    \hat\rho^{n+1}  \\
    \hat {\mathcal{J}}^{n+1}  \\
    \hat\phi^{n+1}   \\
    \hat\psi^{n+1}  \\
    \end{array}
\right)=A
    \left(
    \begin{array}{cccc}
    \hat\rho^*  \\
    \hat {\mathcal{J}}^*  \\
    \hat\phi^n  \\
    \hat\psi^n  \\
    \end{array}
\right),
\end{align}
where
\begin{align*}
B=\left(
\begin{array}{cccc}
    1 & 0 & 0 & 0 \\
    0 & 1 & -ik\Delta t & 0 \\
    0 & 0 & 1 & -\Delta t \\
    0 & 0 & \Delta t & \lambda^2 \\
\end{array}
\right),\quad A=\left(
\begin{array}{cccc}
    1 & 0 & 0 & 0 \\
    0 & 1 & 0 & 0 \\
    0 & 0 & 1 & 0 \\
    \Delta t & 0 & 0 & \lambda^2 \\
\end{array}
\right).
\end{align*}
To prove stability, we need to show that the modulus of each complex-valued
eigenvalue $\mu$ of $B^{-1}A$
is less than or equal to 1. We easily deduce that
$\mu_1=\mu_2=1$ and $\mu_3=\frac{\lambda^2+i\lambda\Delta t}{\lambda^2+\Delta t^2}$, $\mu_4=\frac{\lambda^2-i\lambda\Delta t}{\lambda^2+\Delta t^2}$. Since
\begin{align*}
    |\mu_3|=|\mu_4|=\sqrt{\frac{\lambda^4+\lambda^2\Delta t^2}{\lambda^4+2\lambda^2\Delta t^2+\Delta t^4}}<1,
    \end{align*}
 the linearized scheme satisfies von Neumann's $L^2$ stability
    criterion unconditionally.
\end{proof}

The  $L^2$ stability of the full AP-CSLDG-1 scheme follows from the analysis
of the
operators $\mathcal{S}_{\mathbf{f}}^h$ and $\mathcal{S}_{\mathbf{E}}^h$ in the
last two theorems. In view of its consistency and stability, the proposed
discretization is indeed asymptotic preserving in the quasi-neutral limit by
definition of the weak AP property.

\section{Numerical experiments}

\hspace{4pt}
In this section, we present several numerical experiments to verify the effectiveness of the proposed AP-CSLDG schemes.
The main objectives of the numerical experiments conducted in this work are twofold.
First, we assess the efficiency and accuracy of the AP-CSLDG schemes in the non-quasi-neutral regime of the RVP system, validated through benchmark tests including nonlinear Landau damping and two-stream instability.
Second, we examine the AP properties of the scheme in the quasi-neutral regime.
In this case, the consistency is demonstrated via near equilibrium simulations, while the capability of the scheme to overcome the Debye length restriction is verified through the bump-on-tail instability test.

We will examine whether the numerical solutions conserve the following
integral quantities:
\begin{enumerate}
    \item Mass
    \begin{align*}
        \text{Mass} = \int_{\Omega_v}\int_{\Omega_x}f_h(x,v,t)\mathrm dx\mathrm dv.
    \end{align*}
    \item $L^p$ norm, $1\leq p < \infty$
    \begin{align*}
        \|f_h\|_p = \left(\int_{\Omega_v}\int_{\Omega_x}|f_h(x,v,t)|^p\mathrm dx\mathrm dv \right)^{\frac{1}{p}}.
    \end{align*}
    \item Entropy
    \begin{align*}
        \text{Entropy} = \int_{\Omega_v}\int_{\Omega_x}f_h(x,v,t)\log(f_h(x,v,t))\mathrm dx\mathrm dv.
    \end{align*}
    \item Energy
    \begin{align*}
        \text{Energy} = \frac{1}{2}\int_{\Omega_v}\int_{\Omega_x}f_h(x,v,t)v^2\mathrm dx\mathrm dv+\frac{\lambda^2}{2}\int_{\Omega_x}(\partial_x\phi)^2(x,t)\mathrm dx.
    \end{align*}
\end{enumerate}
Another quantity of interest is the logarithm of electrostatic energy $\text{log}(\varepsilon_p)$, where $\varepsilon_p$ is defined as
\begin{align*}
    \varepsilon_p = \frac{\lambda^2}{2}\int_{\Omega_x}(\partial_x\phi)^2(x,t)\mathrm dx.
\end{align*}


Unless otherwise specified, the time step is chosen as follows in our numerical experiments:
\begin{align}\label{dt}
    \Delta t = \frac{\text{CFL}}{\frac{v_{\text{max}}}{\Delta x}+\frac{\text{max}(|\partial_x\phi|)}{\Delta v}}.
\end{align}

\subsection{Numerical experiments for the RVP system in non-quasi-neutral regime}
To begin, we apply our AP-CSLDG schemes to the RVP system in the non-quasi-neutral regime, including nonlinear Landau damping and two-stream instability. In all tests, the dimensionless Debye length $\lambda$ is set equal to 1.
\subsubsection{Nonlinear Landau damping}
\hspace{4pt}
We test the accuracy of the AP-CSLDG-1 by simulation of nonlinear Landau damping. The initial condition is set to be the following perturbed equilibrium
\begin{align}
    f_0(x,v)=\frac{1}{\sqrt{2\pi}}(1+\alpha\text{cos}(kx))\text{exp}(-\frac{v^2}{2}),
\end{align}
with $\alpha=0.5$ and $k=0.5$. The phase space is $\Omega_x\times\Omega_v = [0,4\pi]\times[-5,5]$.
In this experiment, we set the time step~$\Delta t=\text{CFL}\cdot (\text{min}\{\Delta x,\Delta v\})^{k+1}$~and CFL=0.1 to minimize the temporal error. 
The phase space is discretized with $N_x\times N_v$, where $N_x=N_v$ is doubled from 16 to 128 successively. 

The well-known time reversibility of the VP system is used to test the order of convergence, that is, first compute up to time $T$, then recover the solution at $T$ with reverse velocity field and compare it with the initial condition, which is used as a reference solution. 
We show the $L^2$ error and the corresponding order of convergence for AP-CSLDG-1 $P^k$ in Table~\ref{spatial order}. As expected, convergence of order $k+1$  is observed.
\begin{table}[ht]
    \centering
    \caption{Nonlinear Landau damping: $L^2$ error and spatial order of accuracy of AP-CSLDG-1 $P^k$.}
	\begin{tabular}{*{7}{c}}
		\toprule
		\multirow{2}*{mesh} & \multicolumn{2}{c}{$P^1$} &\multicolumn{2}{c}{$P^2$}&\multicolumn{2}{c}{$P^3$}\\
		\cmidrule(lr){2-3}\cmidrule(lr){4-5}\cmidrule(lr){6-7}
		&$L^2$error & Order & $L^2$error & Order & $L^2$error & Order\\
		\midrule
        $16\times 16$&	6.556E-02&      &  5.446E-02  &        &4.654E-02 &\\
		$32\times 32$&	1.727E-02& 1.92 &  7.265E-03  &  2.91  &3.113E-03 &3.90  \\
		$64\times 64$&	4.366E-03&	1.98 &  9.183E-04  &  2.98  &1.963E-04 &3.99 \\
		$128\times 128$&	1.093E-03&	2.00 &  1.150E-04  &	3.00    &1.229E-05 &	4.00 \\
		\bottomrule
	\end{tabular}
    \label{spatial order}
\end{table}

\subsubsection{Two stream instability I}
\begin{figure}[h!]
    \centering
    \begin{minipage}[t]{0.48\linewidth}
        \centering
        \includegraphics[width=\linewidth]{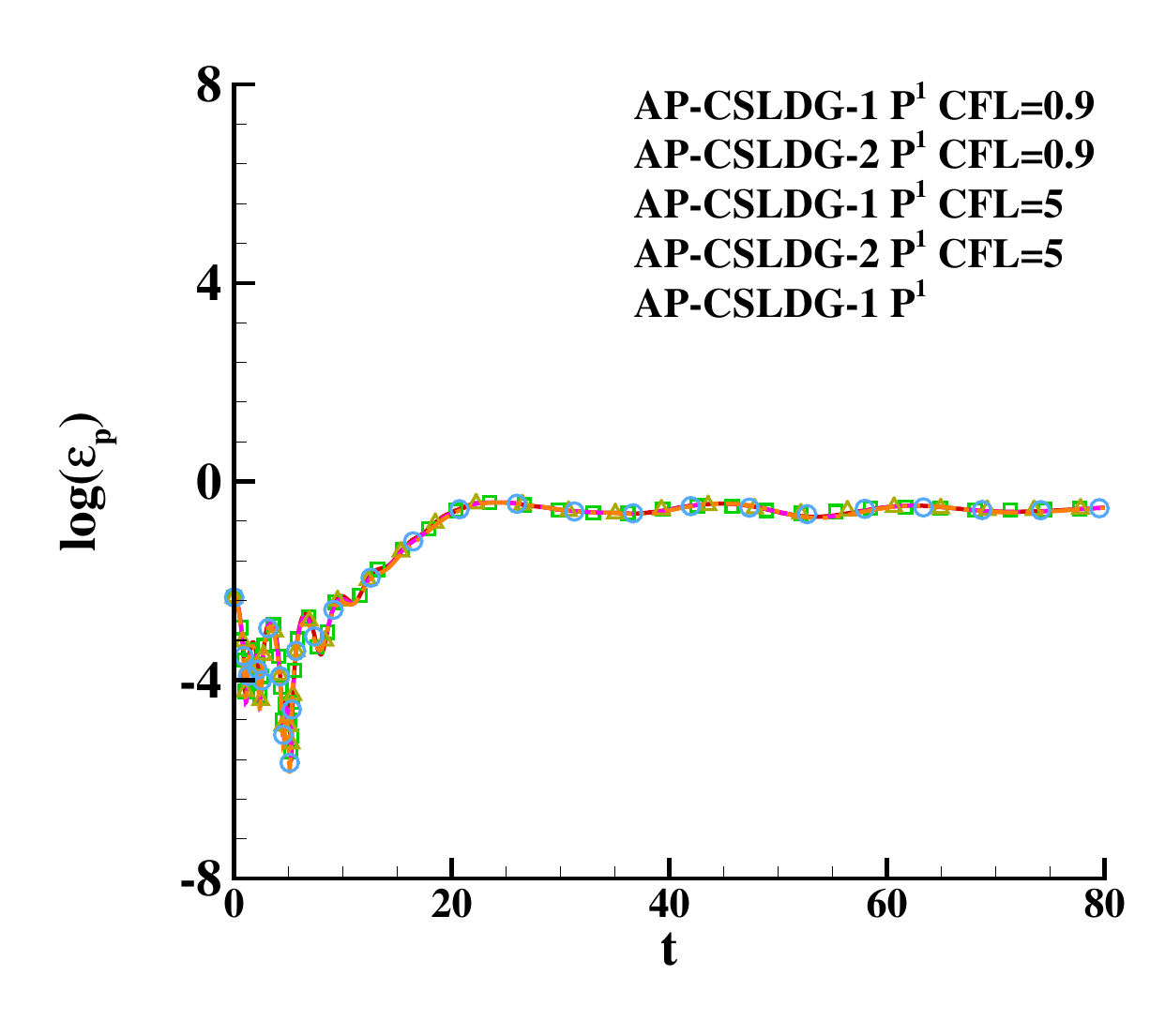}
    \end{minipage}
    \hfill 
    \begin{minipage}[t]{0.48\linewidth}
        \centering
        \includegraphics[width=\linewidth]{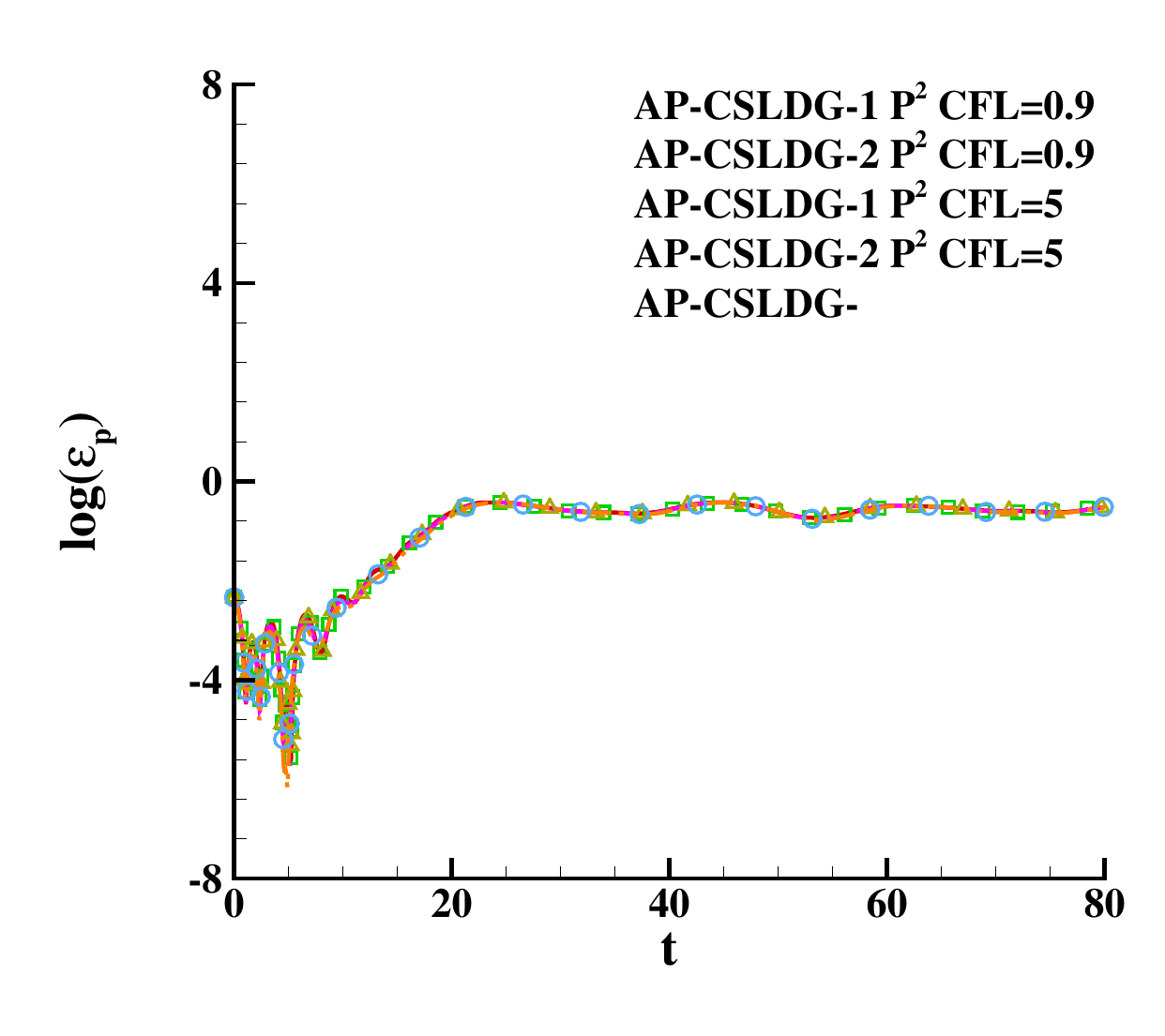}
    \end{minipage}
    \caption{Two stream instability I: $P^1$ (left) and $P^2$ (right) polynomial spaces, time evolution of the logarithm of electrostatic energy log$(\varepsilon_p)$ with different CFL numbers.} 
    \label{fig:ex3EL2}
\end{figure}
\begin{figure}[h!]
    \centering
    \begin{minipage}[t]{0.48\linewidth}
        \centering
        \includegraphics[width=\linewidth]{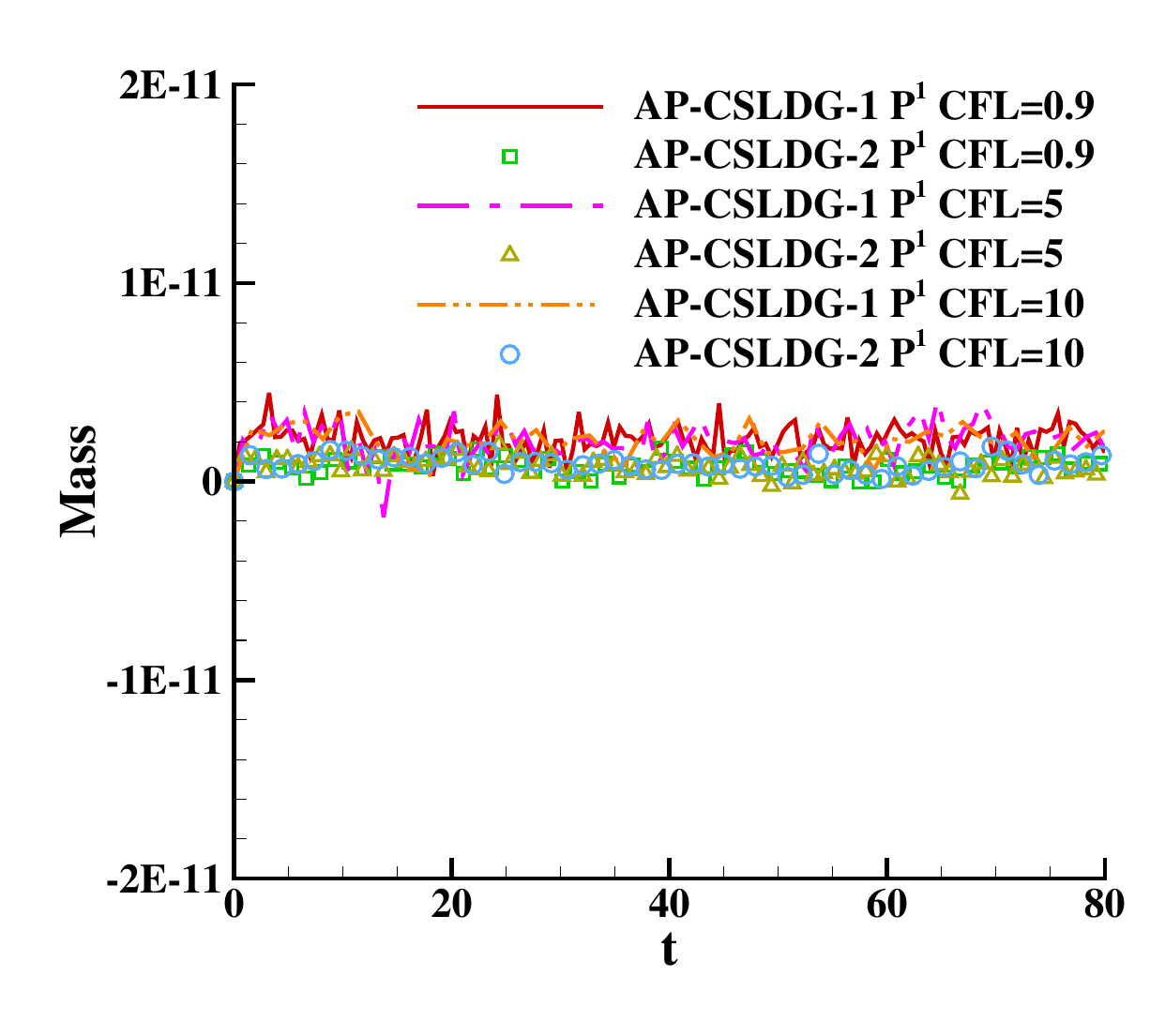}
    \end{minipage}
    \hfill 
    \begin{minipage}[t]{0.48\linewidth}
        \centering
        \includegraphics[width=\linewidth]{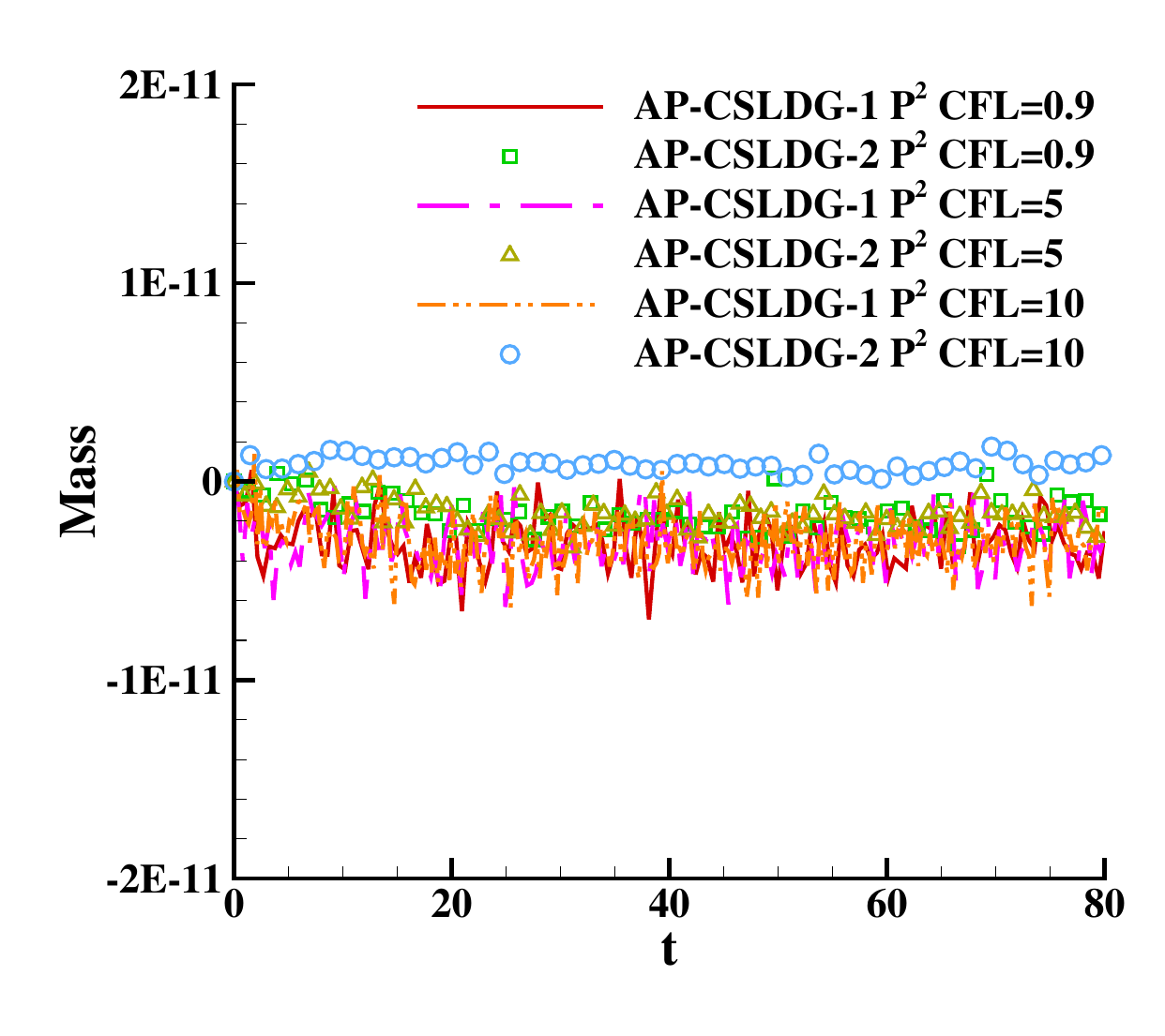}
    \end{minipage}
    
    \vspace{1em} 
    
    \begin{minipage}[t]{0.48\linewidth}
        \centering
        \includegraphics[width=\linewidth]{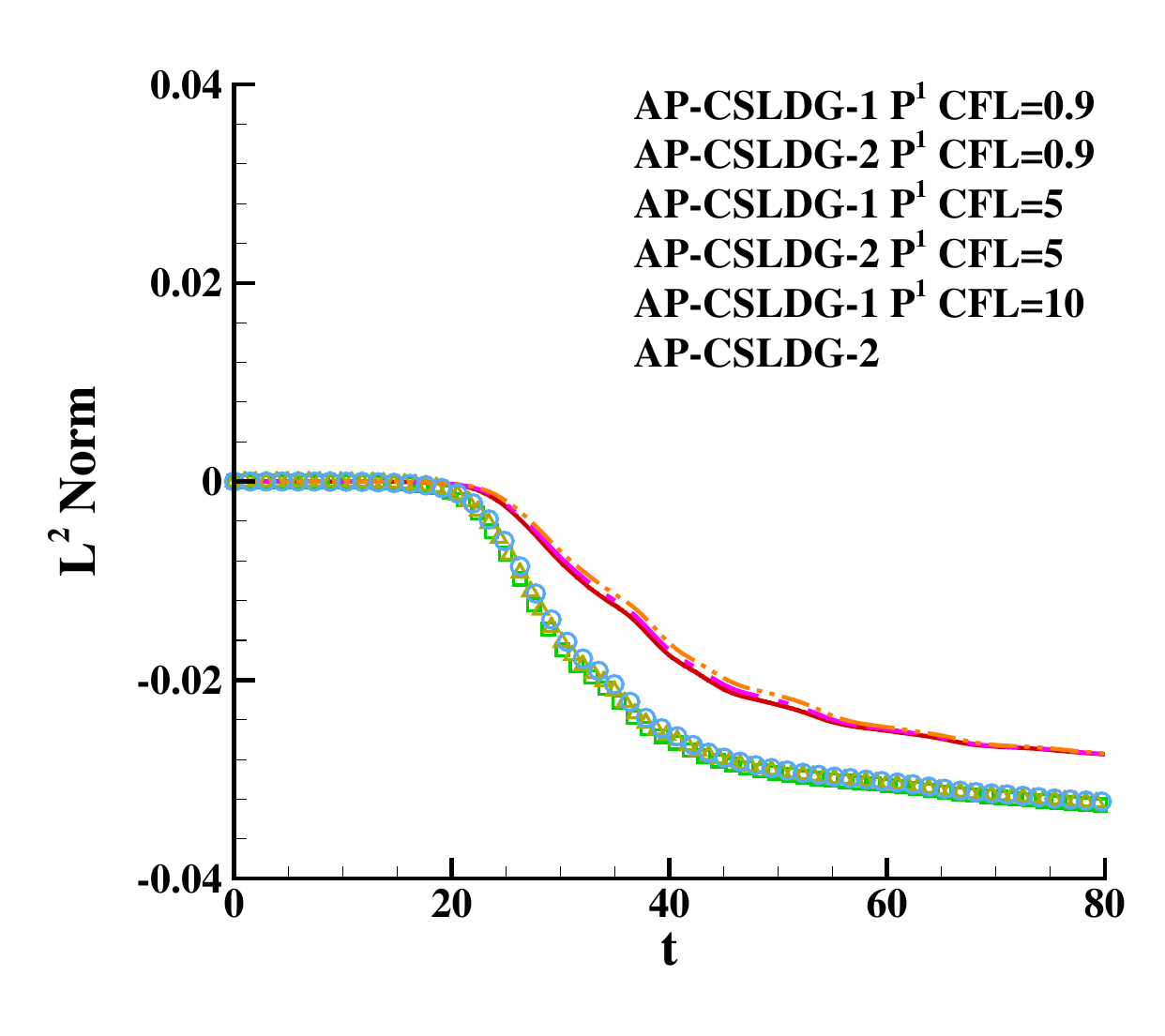}
    \end{minipage}
    \hfill
    \begin{minipage}[t]{0.48\linewidth}
        \centering
        \includegraphics[width=\linewidth]{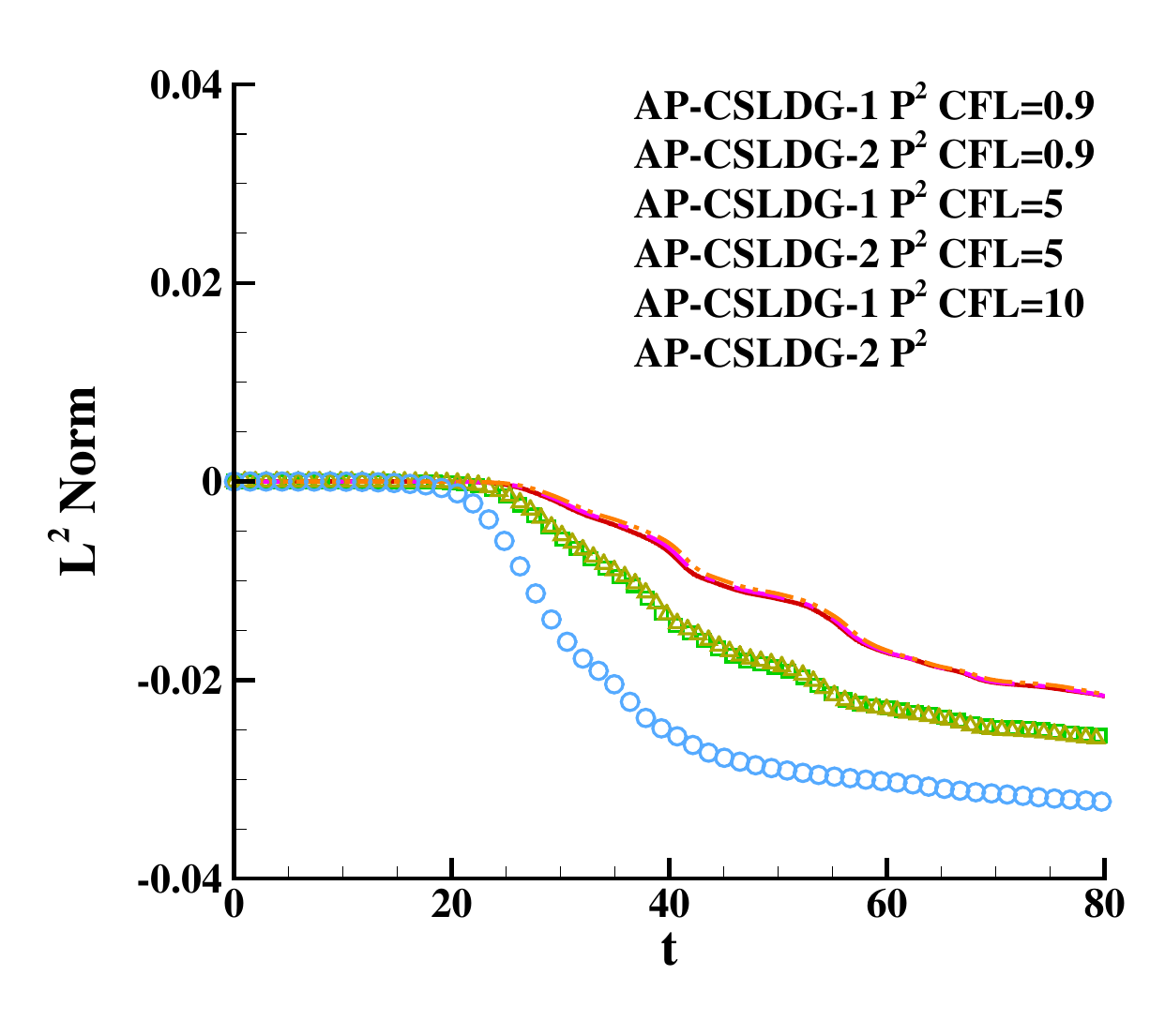}
    \end{minipage}
    
    \caption{Two stream instability I: $P^1$ (left) and $P^2$ (right) polynomial spaces, time evolution of the relative deviations for mass (top) and $L^2$ norm (bottom) with different CFL numbers.} 
    \label{fig:ex3Conser_1}
\end{figure}
\begin{figure}[h!]
    \centering    
    \begin{minipage}[t]{0.48\linewidth}
        \centering
        \includegraphics[width=\linewidth]{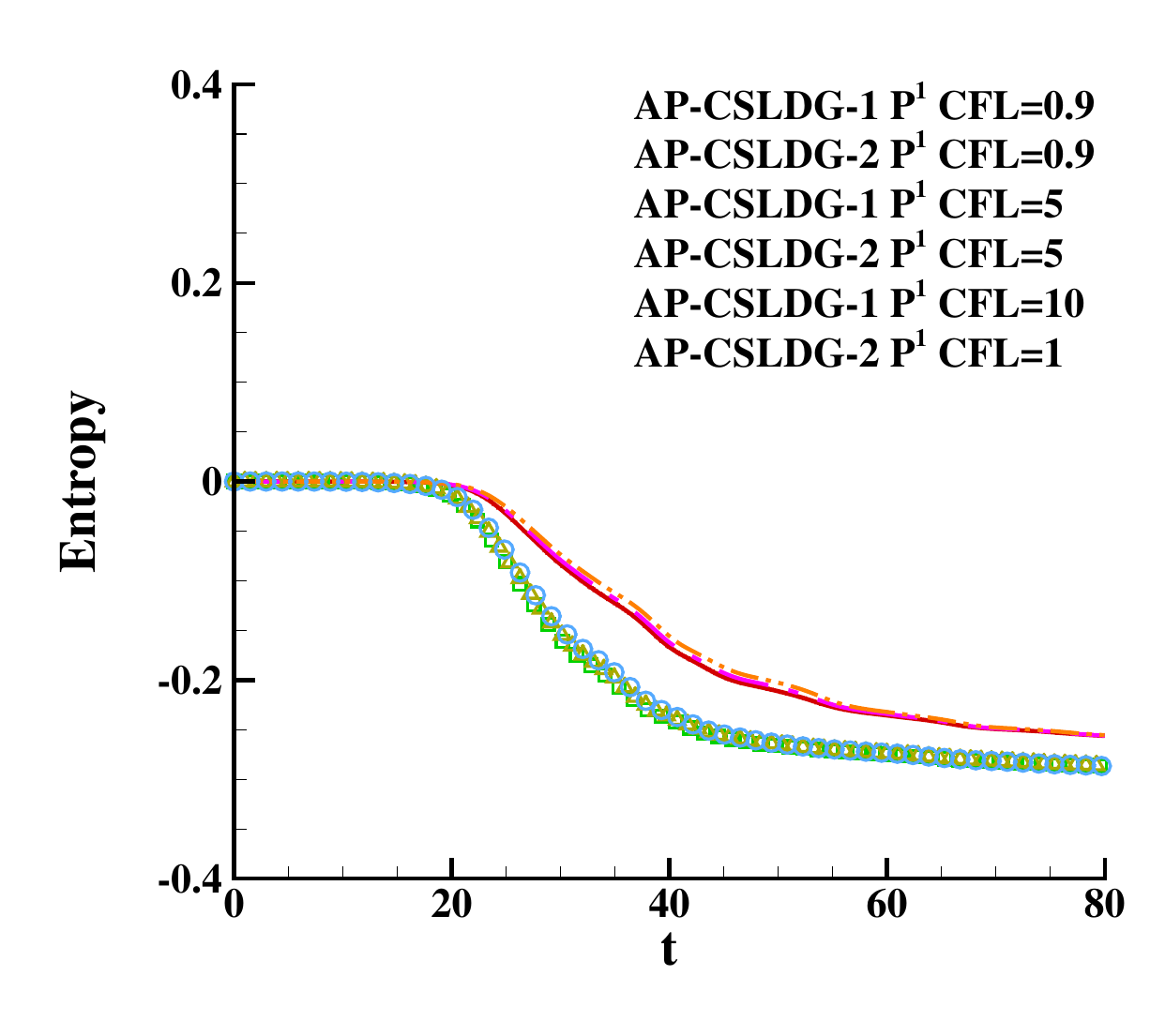}
    \end{minipage}
    \hfill
    \begin{minipage}[t]{0.48\linewidth}
        \centering
        \includegraphics[width=\linewidth]{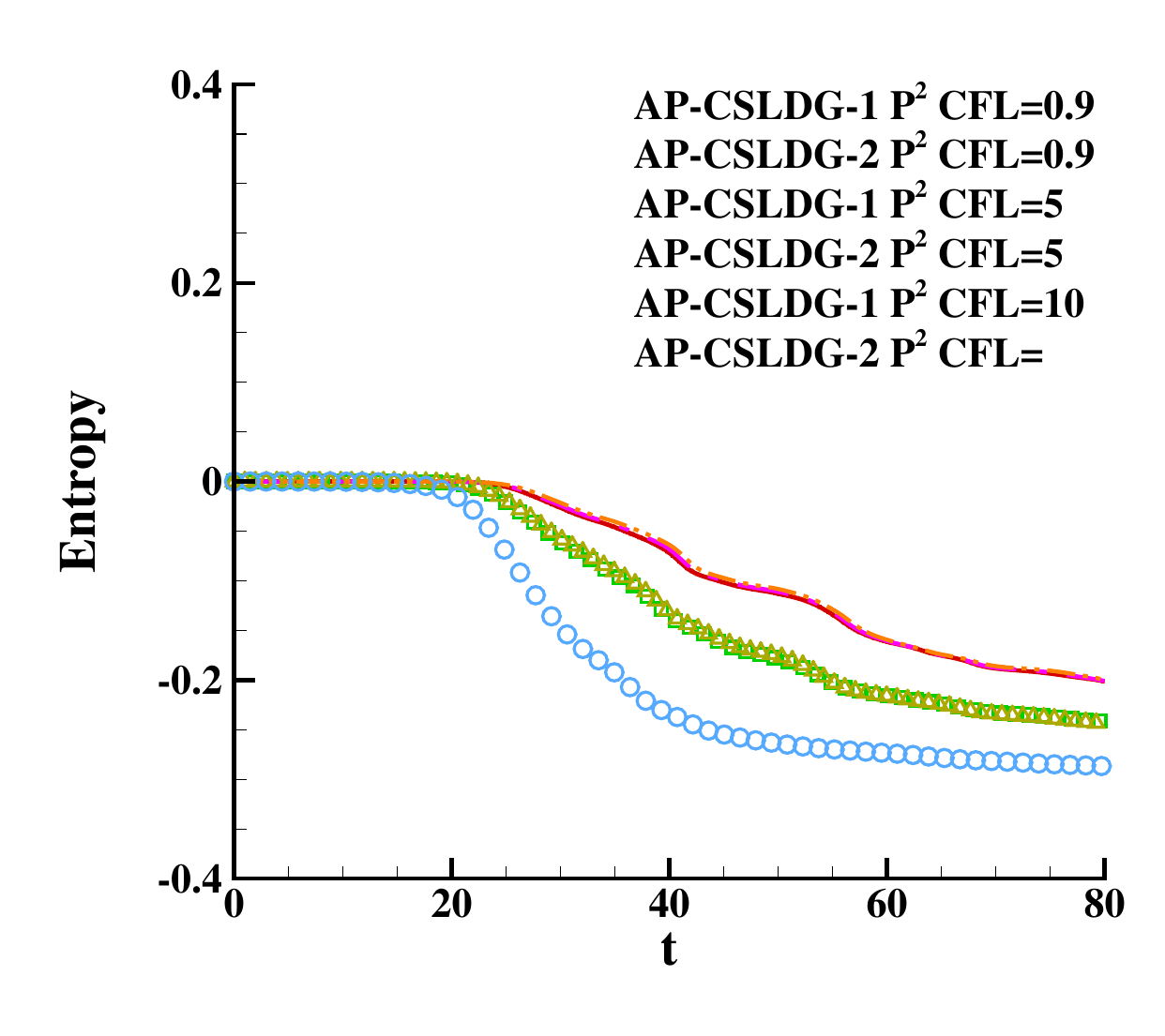}
    \end{minipage}

    \vspace{1em} 
    
    \begin{minipage}[t]{0.48\linewidth}
        \centering
        \includegraphics[width=\linewidth]{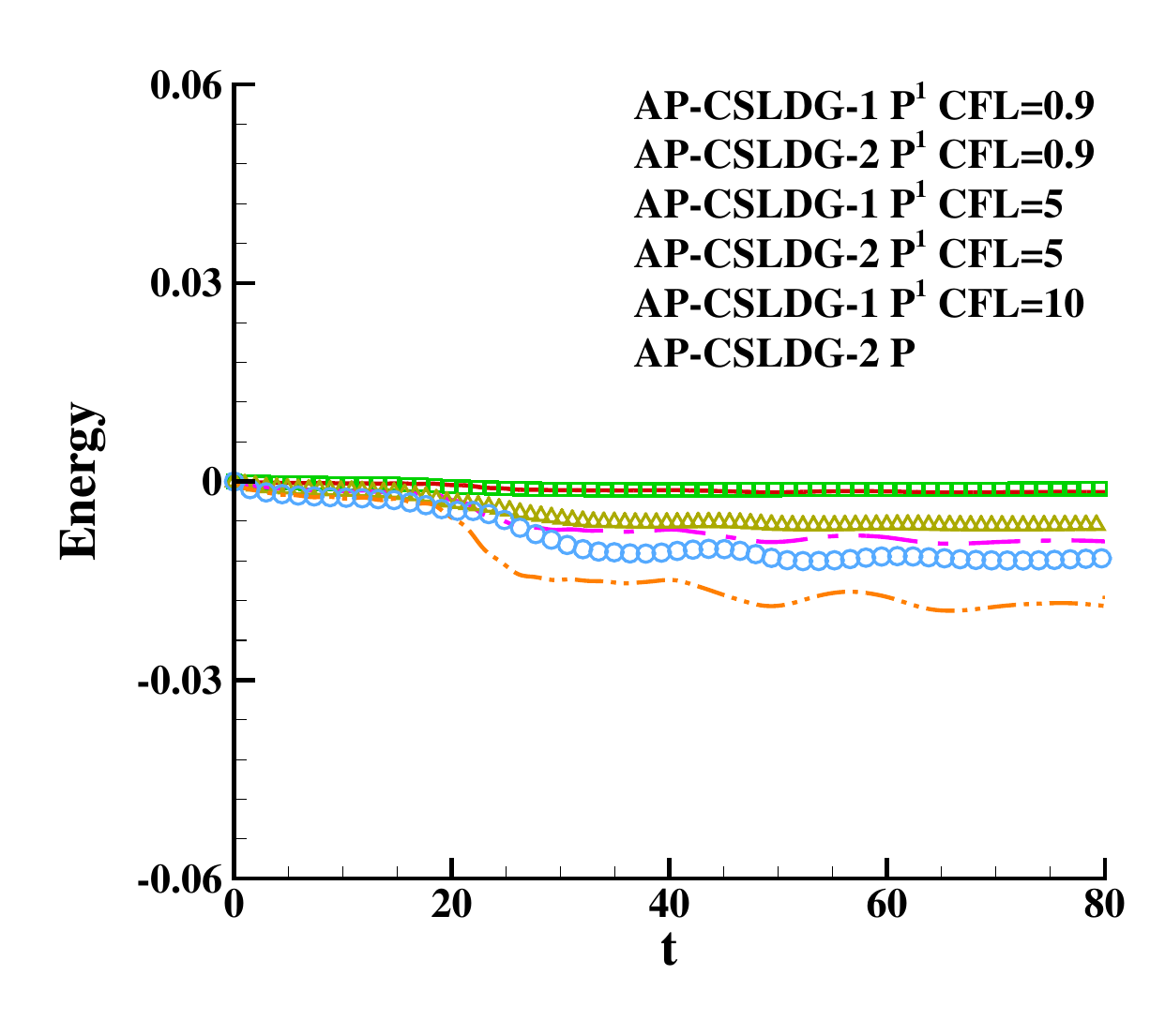}
    \end{minipage}
    \hfill
    \begin{minipage}[t]{0.48\linewidth}
        \centering
        \includegraphics[width=\linewidth]{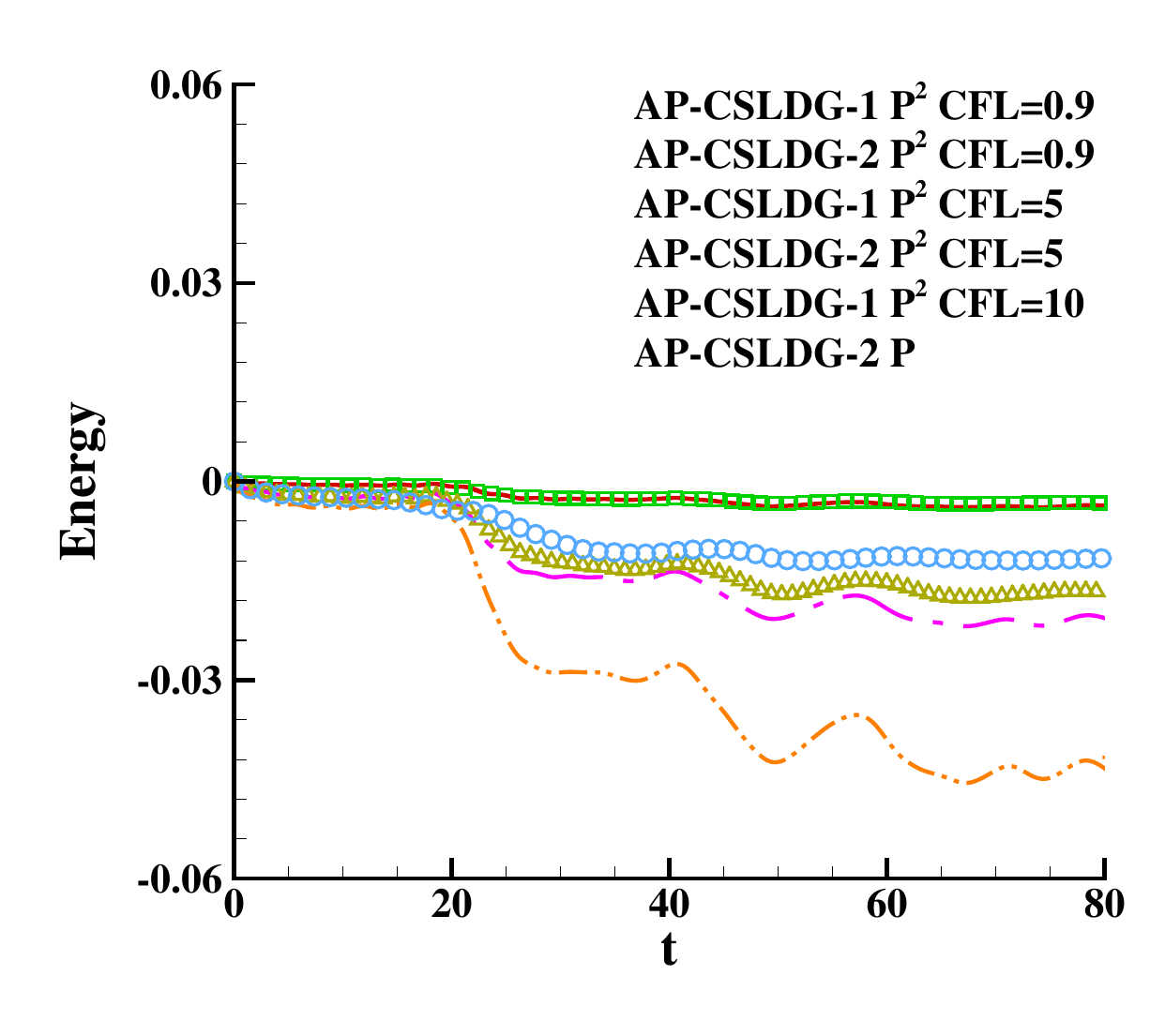}
    \end{minipage}

    \caption{Two stream instability I: $P^1$ (left) and $P^2$ (right) polynomial spaces, time evolution of the relative deviations for entropy (top) and energy (bottom) with different CFL numbers.} 
    \label{fig:ex3Conser_2}
\end{figure}
\begin{figure}[h!]
    \centering    
    \begin{minipage}[t]{0.48\linewidth}
        \centering
        \includegraphics[width=\linewidth]{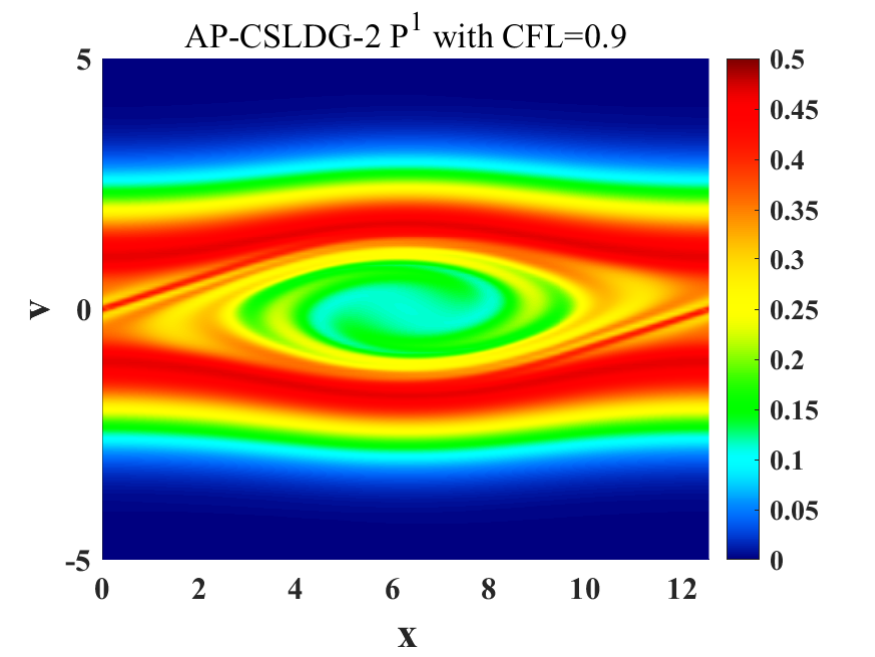}
    \end{minipage}
    \hfill
    \begin{minipage}[t]{0.48\linewidth}
        \centering
        \includegraphics[width=\linewidth]{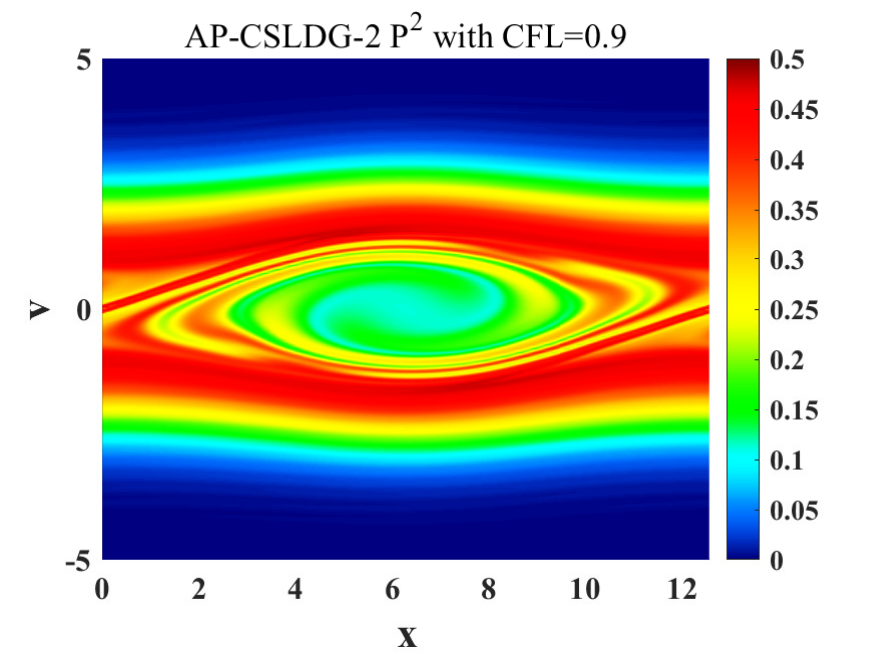}
    \end{minipage}

    \vspace{1em} 
    
    \begin{minipage}[t]{0.48\linewidth}
        \centering
        \includegraphics[width=\linewidth]{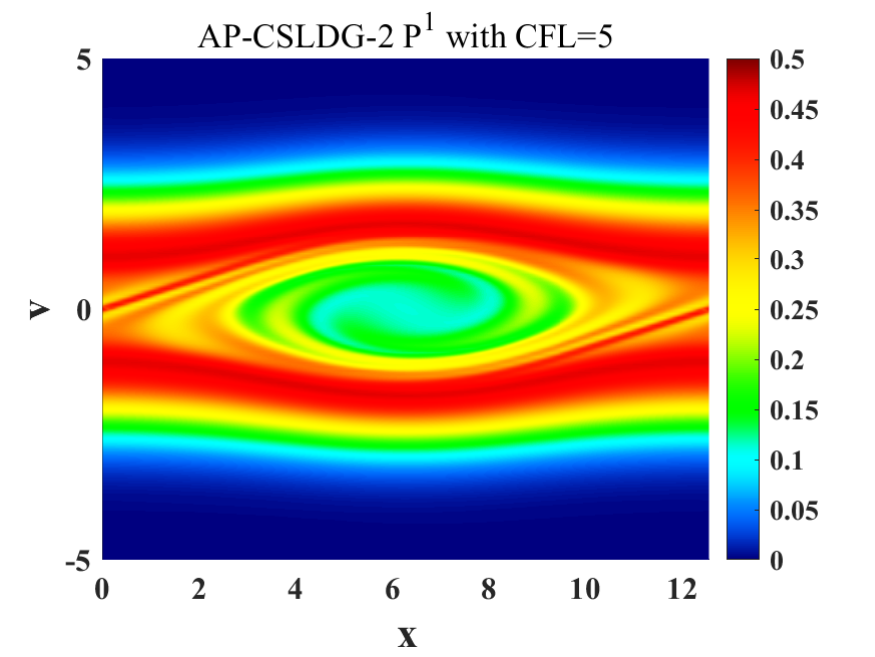}
    \end{minipage}
    \hfill
    \begin{minipage}[t]{0.48\linewidth}
        \centering
        \includegraphics[width=\linewidth]{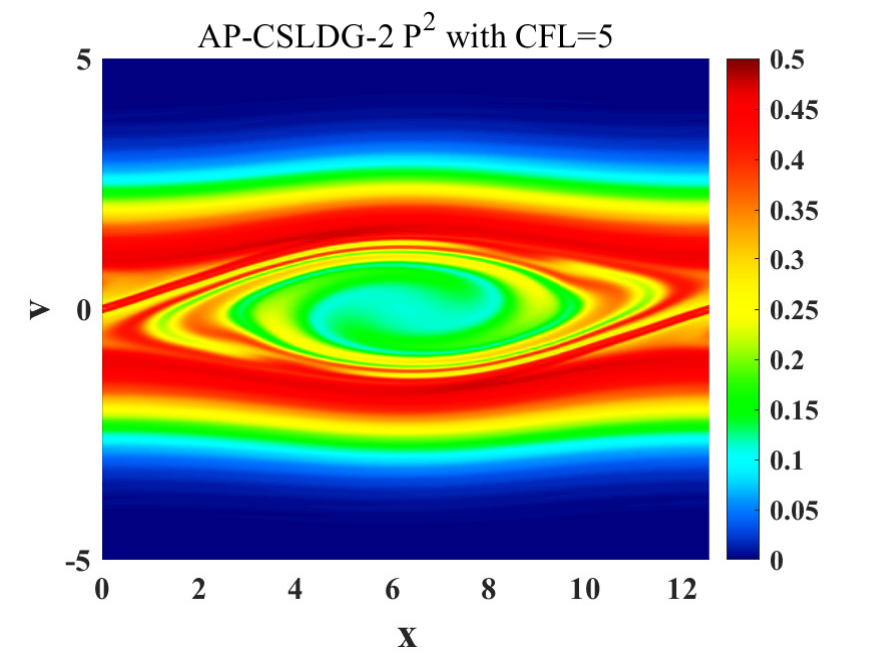}
    \end{minipage}

    \vspace{1em} 
    
    \begin{minipage}[t]{0.48\linewidth}
        \centering
        \includegraphics[width=\linewidth]{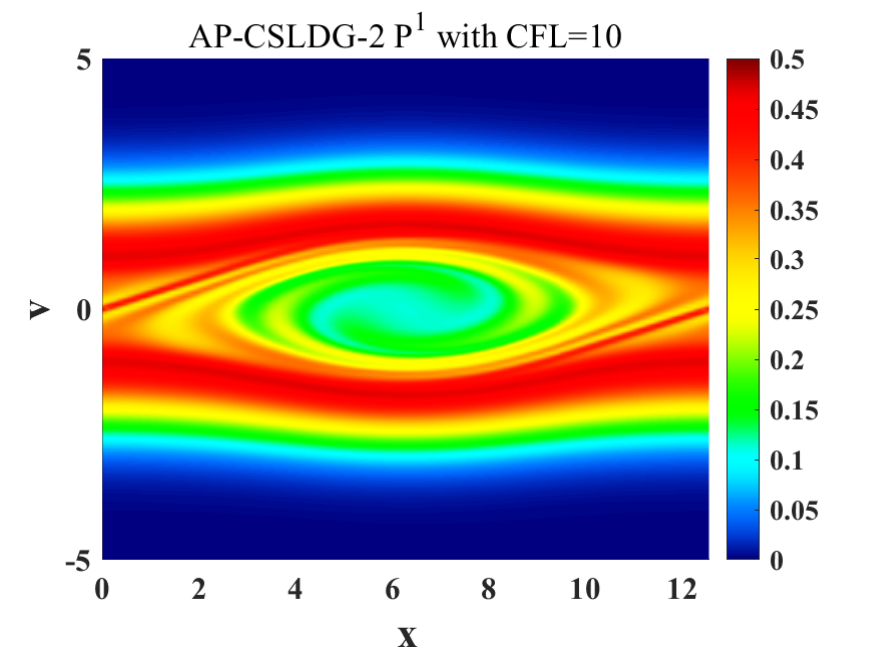}
    \end{minipage}
    \hfill
    \begin{minipage}[t]{0.48\linewidth}
        \centering
        \includegraphics[width=\linewidth]{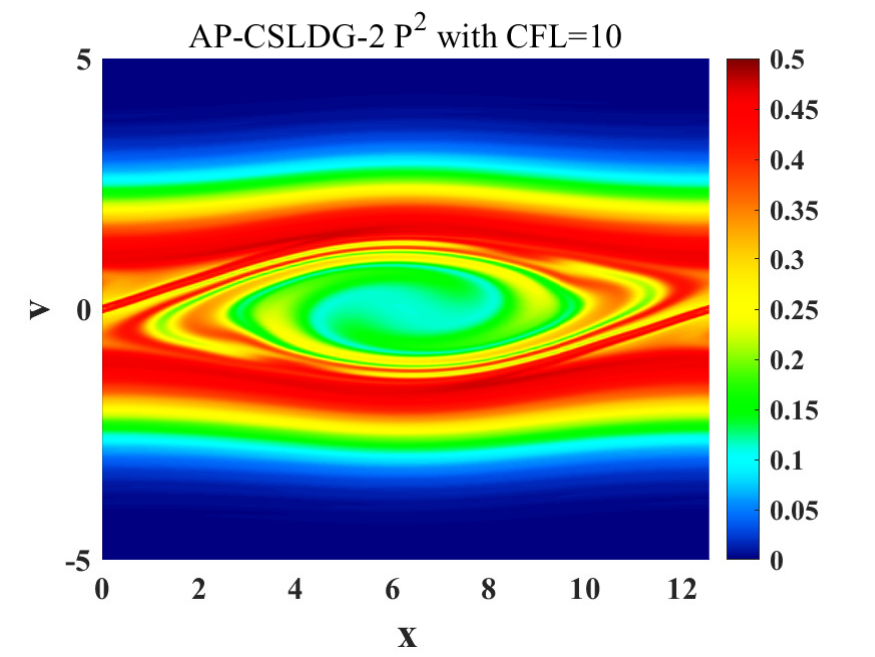}
    \end{minipage}

    \caption{Two stream instability I: The distribution function simulated by AP-CSLDG-2 scheme with different CFL numbers at $T=53$.} 
    \label{fig:ex3function}
\end{figure}
\hspace{4pt}
Consider the symmetric warm two stream instability problem, in which the electron distribution function of the VP system starts with the following unstable initial condition: 
\begin{align}\label{Instability1}
    f_0(x,v)=\frac{2}{7\sqrt{2\pi}}(1+5v^2)\left(1+\alpha((\text{cos}(2k x)+\text{cos}(3k x))/1.2+\text{cos}(k x))\right)\text{exp}\left(-\frac{v^2}{2}\right)
\end{align}
with $\alpha=0.01$ and $k=0.5$. The phase space is $\Omega_x\times\Omega_v = [0,2\pi/k]\times[-10,10]$ and is discretized with $256\times 256$ cells in this experiment. The final time $T$ is set to $80$.

The AP-CSLDG schemes with $P^1$ and $P^2$ polynomial spaces deliver consistent electrostatic energy results for different CFL numbers, as shown in Fig.~\ref{fig:ex3EL2}.
Mass conservation is also maintained by the AP-CSLDG schemes for different CFL numbers, as shown in Fig.~\ref{fig:ex3Conser_1}. The results presented in
Fig.~\ref{fig:ex3Conser_1} and Fig.~\ref{fig:ex3Conser_2} show that the $L^2$ norm and relative changes in entropy behave correctly. 
While the total energy of the VP system is not strictly conserved, it remains stable as the CFL number increases.

Inspecting the distribution function in the phase space at $T=53$, as illustrated in Fig.~\ref{fig:ex3function}, we find that the AP-CSLDG-2 scheme exhibits same behavior for different CFL numbers.

Through this case, we demonstrate that our AP-CSLDG schemes achieve consistent and stable simulation results across varying CFL numbers.

\subsubsection{Two stream instability II}
\hspace{4pt}

Let us now run the symmetric two stream instability test. Similarly to  \cite{CROUSEILLES2010}, we initialize $f$ by
\begin{align}\label{Instability2}
    f_0(x,v)=\frac{1}{2v_{\text{th}}\sqrt{2\pi}}\left(1+0.05\text{cos}(kx)\right)\left[\text{exp}\left(-\frac{(v-\text{u})^2}{2v^2_{\text{th}}}\right)+\text{exp}\left(-\frac{(v+\text{u})^2}{2v^2_{\text{th}}}\right)\right]
\end{align}
with u=0.99, ~$v_{\text{th}}=0.3$~and~$k=\frac{2}{13}$. 
The phase space is $\Omega_x\times\Omega_v = [0,2\pi/k]\times[-5,5]$. 
Unless otherwise specified, the CFL number is set to be 3 and the phase space is discretized with $N_x\times N_v$ cells, where ~$N_x=64,N_v=128$ or $N_x=128,N_v=256$~ in this experiment.

We test the AP-CSLDG-2 scheme with $P^2$ and $P^3$ polynomial spaces. 
When the numerical simulation reaches $T=40$, as shown in Fig.~\ref{fig:ex4function@t40}, it can be observed that higher order polynomial spaces capture more detailed physical structure of the distribution function. 


\begin{figure}[h!]
    \centering    
    \begin{minipage}[t]{0.48\linewidth}
        \centering
        \includegraphics[width=\linewidth]{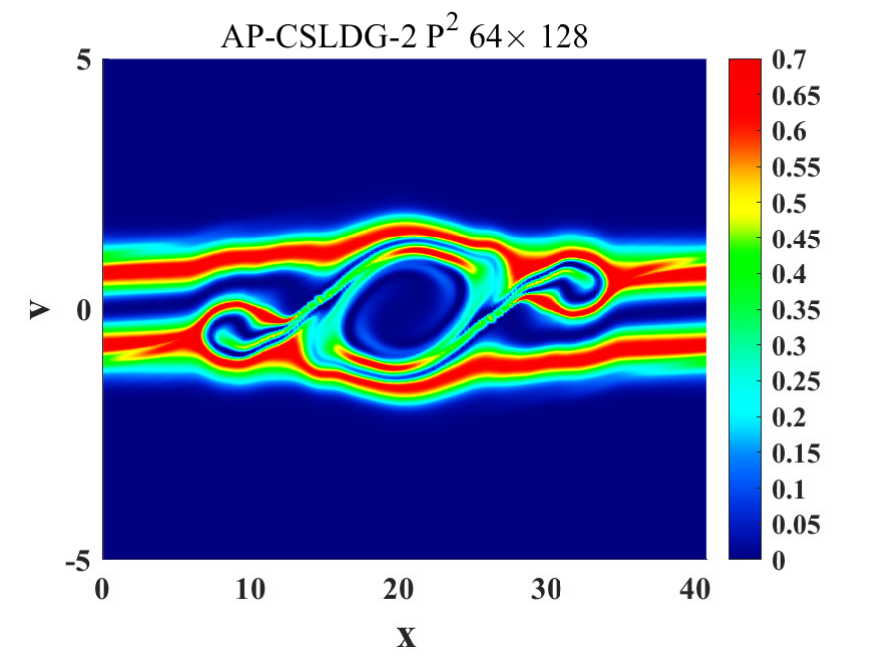}
    \end{minipage}
    \hfill
    \begin{minipage}[t]{0.48\linewidth}
        \centering
        \includegraphics[width=\linewidth]{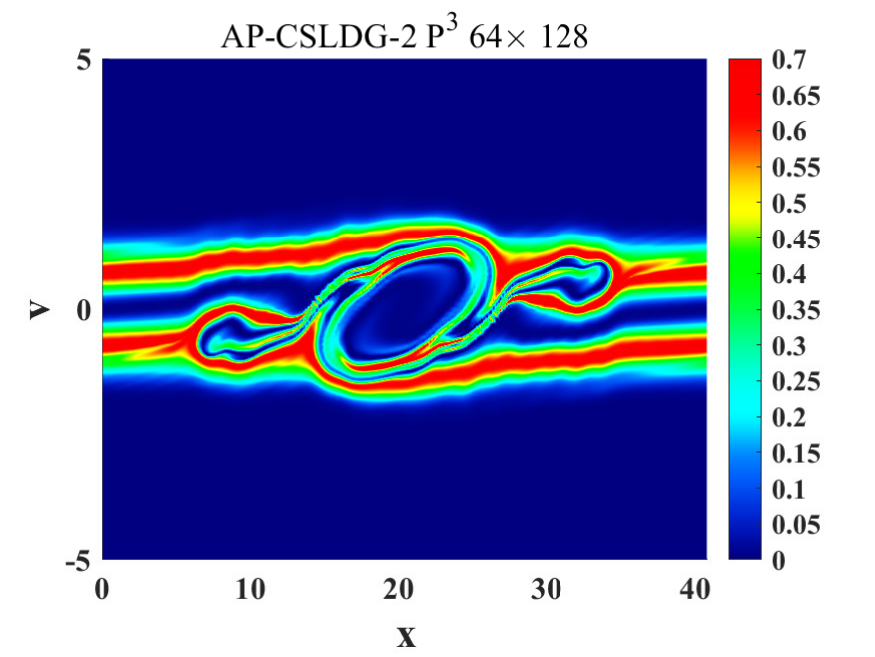}
    \end{minipage}

    \vspace{1em} 
    
    \begin{minipage}[t]{0.48\linewidth}
        \centering
        \includegraphics[width=\linewidth]{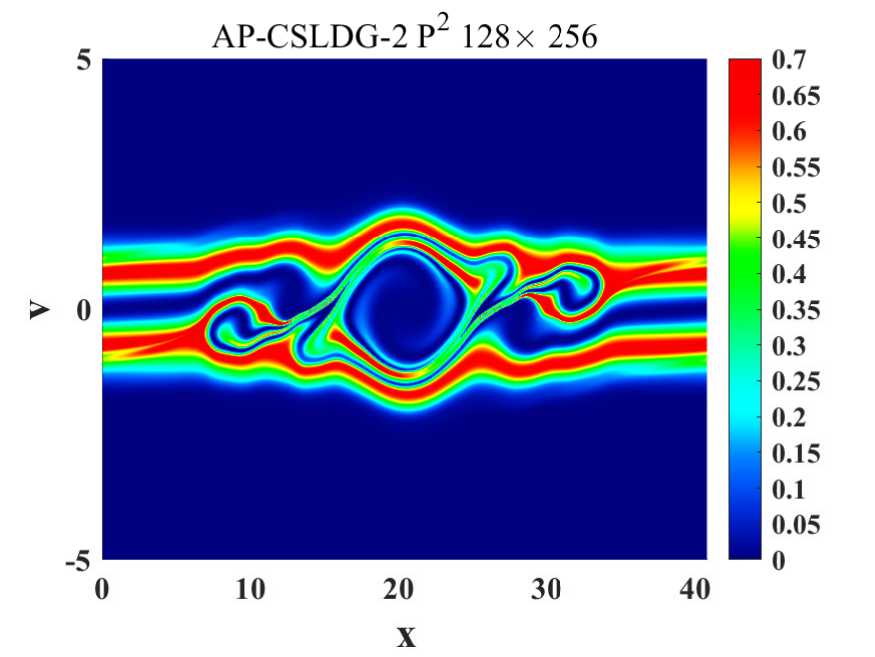}
    \end{minipage}
    \hfill
    \begin{minipage}[t]{0.48\linewidth}
        \centering
        \includegraphics[width=\linewidth]{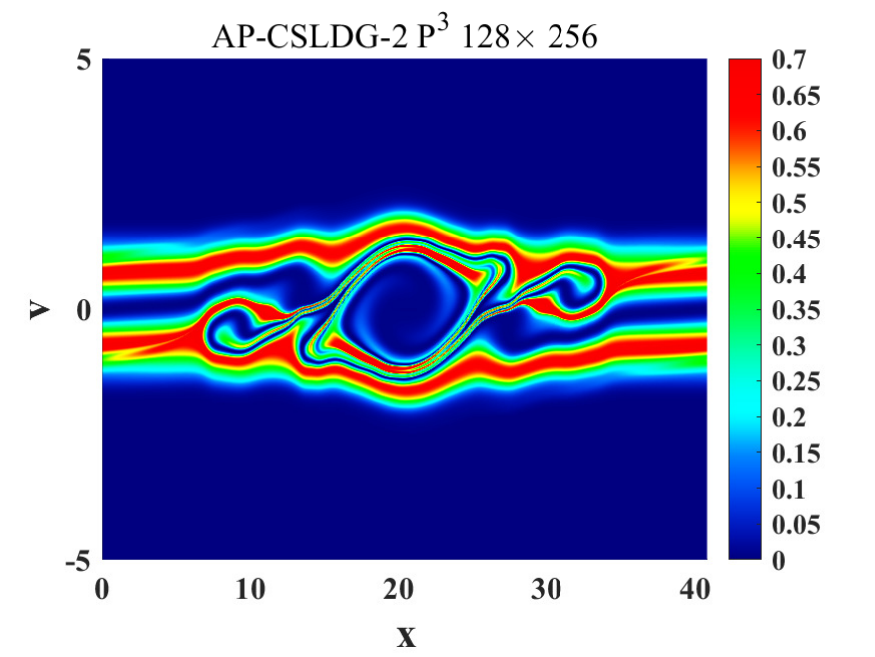}
    \end{minipage}

    \caption{Two stream instability II: $P^2$ (left) and $P^3$ (right) polynomial spaces, the distribution function simulated by AP-CSLDG-2 scheme with different discretization at $T=40$ with CFL=3.} 
    \label{fig:ex4function@t40}
\end{figure}
\subsection{Numerical experiments for the RVP system in quasi-neutral regime}
\hspace{4pt}
Let us now verify that the AP-CSLDG-1 scheme possesses the AP properties in the quasi-neutral limit via numerical experiments that include near equilibrium and bump-on-tail instability simulations.
\subsubsection{Near equilibrium}
\hspace{4pt}
We first consider the initial distribution function
\begin{align*}
    f_0(x,v)=\frac{1}{\sqrt{2\pi}}(1+\alpha\cos(\frac{x}{2}))\text{exp}(-\frac{v^2}{2}).
\end{align*}
The corresponding initial density is 
\begin{align*}
    \rho(x,0)=1+\alpha\cos(\frac{x}{2})\text{erf}(\frac{v_m}{\sqrt{2}}),
\end{align*}
where $\text{erf}(z)=\frac{2}{\sqrt{\pi}}\int_0^ze^{-t^2}\mathrm dt$ is the error function. 
The initial momentum is thus given by
\begin{align*}
    (\rho u)(x,0)=0.
\end{align*}

If we set $\alpha=0$ ($\alpha=10^{-16}$ in numerical experiments) and run simulations in $\Omega_x\times\Omega_v = [0,4\pi]\times[-12,12]$, then the initial condition can be regarded as satisfying the quasi-neutral state constraint, i.e., $\rho(x,0)=1$ and $\partial_x(\rho u)(x,0)=0$.
In particular, the electrostatic potential is updated by
\begin{align*}
    \partial_x(\rho^*\partial_x\phi)=\partial_{xx}S^n,
\end{align*}
as assumed in Theorem~\ref{theorem:consistency}.
The phase space is discretized with $128\times 128$ cells and the final time is $T=80$. 
We test the AP-CSLDG-1 scheme with $P^1$ polynomial space.

We control the time step through the CFL number in Eq.~\eqref{dt} to examine how well the quasi-neutral state is preserved under different time step sizes. 
The results as shown in Fig.~\ref{fig:Near equilibrium} indicate the existence of a critical CFL value (is approximately 6) : when the CFL number is less than or equal to this threshold, the electron change density deviates from 1 only within machine precision in the $L^2$ norm; 
however, when the CFL number exceeds this critical value, the error accumulates over time and eventually stabilizes.
In the future work, we will conduct a more detailed investigation of this phenomenon.
\begin{figure}[h!]
    \centering
    \includegraphics[width=0.48\linewidth]{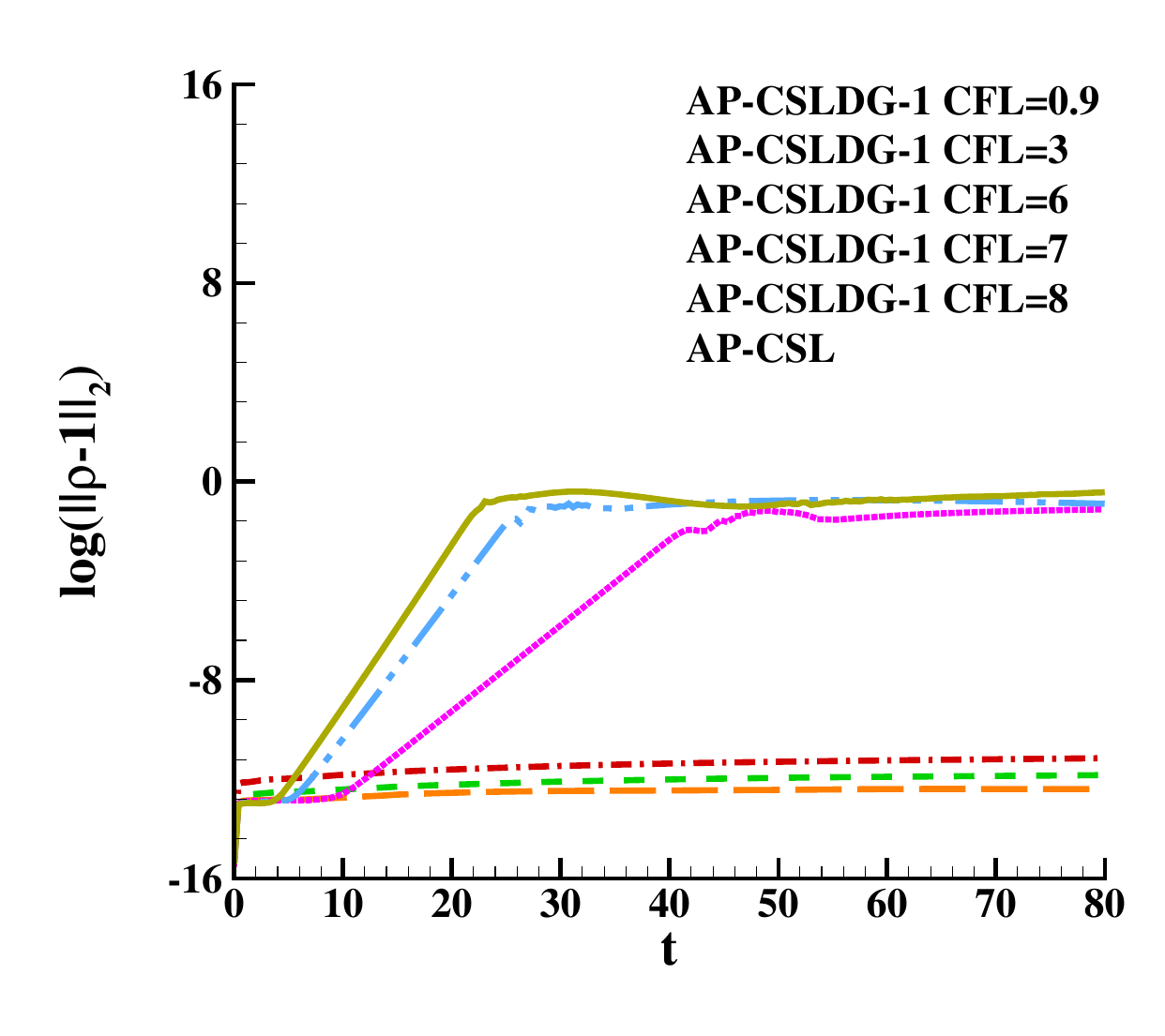}
    \caption{Near equilibrium: The logarithm of $\Vert \rho-1\Vert_{L^2}$ simulated by AP-CSLDG-1 scheme with different CFL numbers.} 
    \label{fig:Near equilibrium}
\end{figure}

\subsubsection{Bump-on-tail instability}
\hspace{4pt}
Finally, we verify that the AP-CSLDG-1 scheme possesses AP properties by simulating bump-on-tail instability in quasi-neutral limit. 
The initial condition is set to be
\begin{align}\label{Instability3}
    f_0(x,v)=f_p(v)(1+\alpha \text{cos}(kx))
\end{align}
with~$k=0.3$,~$\alpha=0.04(0.01+0.99\lambda)$. The distribution function~$f_p$~is a bump of Maxwell distribution on its tail, which is given by
\begin{align*}
    f_p(v)=n_p\text{exp}(-\frac{v^2}{2})+n_b\text{exp}(-\frac{(v-u)^2}{2RT})
\end{align*}
with~$n_p=\frac{9}{10\sqrt{2\pi}}$~,~$n_p=\frac{2}{10\sqrt{2\pi}}$~,~$u=4.5$~and~$RT=0.25$. 
The phase space is $\Omega_x\times\Omega_v = [0,L]\times[v_{\text{min}},v_{\text{max}}]$, where ~$L=2\pi/k$. 
For all simulations, the domain is discretized with $N_x\times N_v=256\times256$, and $v_{\text{min}}=-6$,~$v_{\text{max}}=9$~for~$\lambda >10^{-3}$, 
while $v_{\text{min}}=-12$,~$v_{\text{max}}=12$~for~$\lambda\leq 10^{-3}$. 

Consider the case in non-quasi-neutral regime ($\lambda=1$), where the spatial size and time step resolve the dimensionless Debye length, i.e., $\Delta x=0.16\lambda$ and $\Delta t \approx 0.008\lambda$.
As shown in Fig.~\ref{fig:ex5lambda1}, the AP-CSLDG-1 scheme and the reference CSLDG scheme produce similar results in terms of electrostatic energy and time stepping.

When $\lambda=0.1$, we take $N_x=256,128,64$ so that $\Delta x$ gradually exceeds the Debye length. 
Additionally, we set CFL=0.9, controlling $\Delta t < \lambda$, which still resolves the Debye length in time.
The $P^2$ polynomial space is used in this experiment.
The comparison between the AP-CSLDG-1 scheme and the reference CSLDG scheme is shown in Fig.~\ref{fig:ex5lambda01_gd}. 
From the left panel, it can be seen that even after $\Delta x$ exceeds the Debye length, the electrostatic energy remains consistent for both schemes.
The right panel shows that the time step continues to resolve the Debye length.
Furthermore, due to Eq.~\eqref{dt}, $\Delta t$ also reflects the change in the electric field, indicating that the electric field remains stable.

Since the AP-CSLDG and CSLDG schemes are not restricted by the CFL condition, we can increase the CFL number to allow the time step size to exceed the Debye length. 
With the spatial step size fixed at $\Delta x=6.5\lambda~(N_x=32)$, we select CFL = 0.9, 3, and 5, respectively. 
As shown in Fig.~\ref{fig:ex5lambda01_bd}, for CFL = 5 the electrostatic energy $\varepsilon_p$ obtained by the CSLDG scheme reaches an order of $10^2$ within a short time. 
Simultaneously, due to Eq.~\eqref{dt}, the time step size $\Delta t$ decreases rapidly. 
When the electrostatic field becomes too large, the distribution function is pushed beyond the velocity space, causing the CSLDG scheme to fail\cite{LIU2020}.
Fig.~\ref{fig:ex5lambda01_gdd} shows the computational results of the AP-CSLDG-1 method are still reasonable under the same circumstances. 
Even though the time step exceeds the Debye length, the numerical solution
remains consistent and stable. 
This confirms that the AP-CSLDG-1 scheme recovers the correct quasi-neutral limit without requiring refined meshes or small time steps.

Finally, to verify the AP properties, simulations of the VP system are conducted by AP-CSLDG-1 scheme for Debye lengths $\lambda=10^{-3},10^{-6},0$, based on a fixed discretization with $N_x=N_v=256$ and a fixed time step $\Delta t=10^{-3}$. 
As shown in the left panel of Fig.~\ref{fig:ex5quasi}, the electrostatic field energy of the system converges to the value that it has in the case $\lambda=0$. 
Furthermore, the right panel shows that the AP-CSLDG-1 scheme maintains mass conservation well for different Debye lengths.
\begin{figure}[h!]
    \centering    
    \begin{minipage}[t]{0.48\linewidth}
        \centering
        \includegraphics[width=\linewidth]{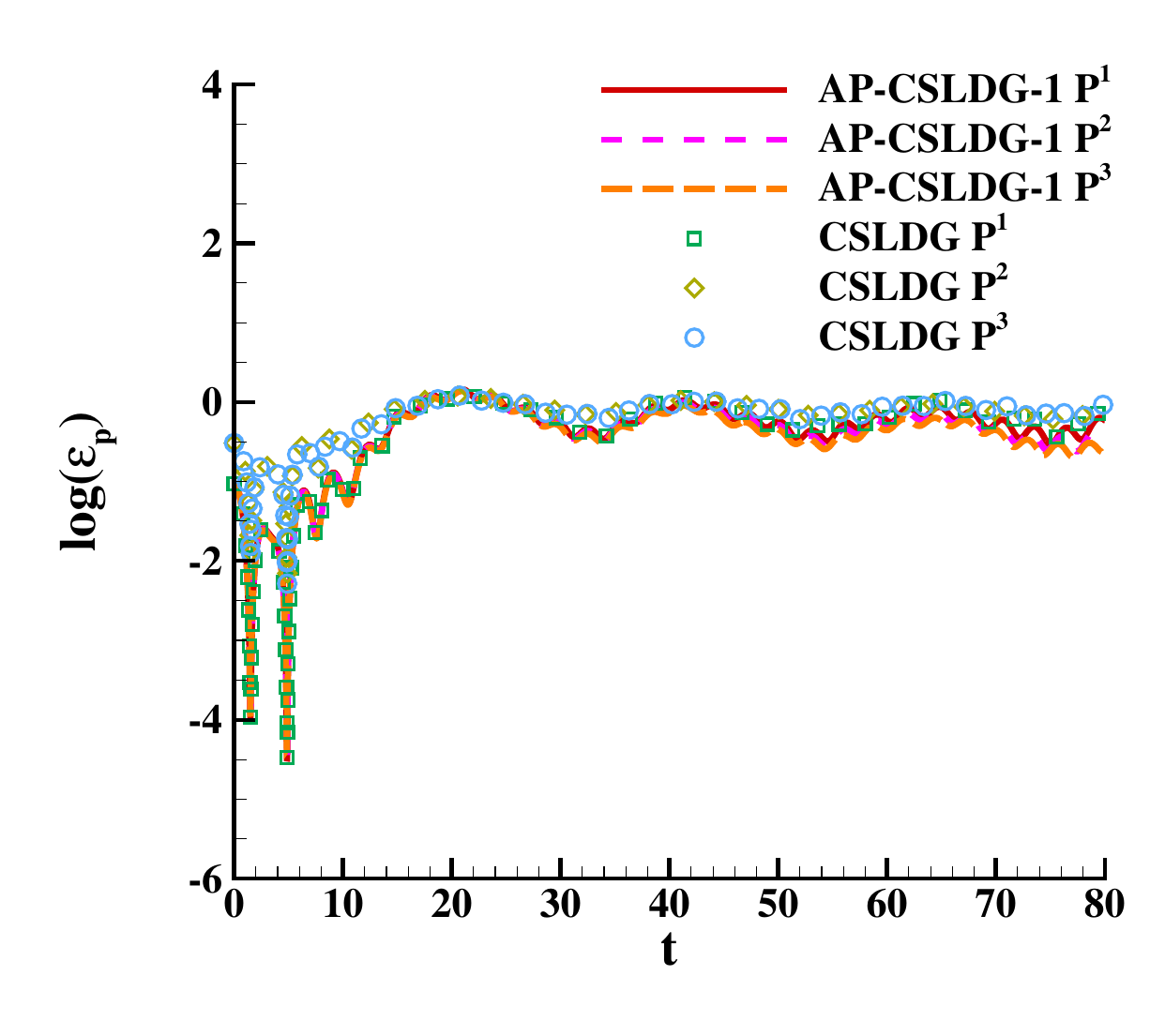}
    \end{minipage}
    \hfill
    \begin{minipage}[t]{0.48\linewidth}
        \centering
        \includegraphics[width=\linewidth]{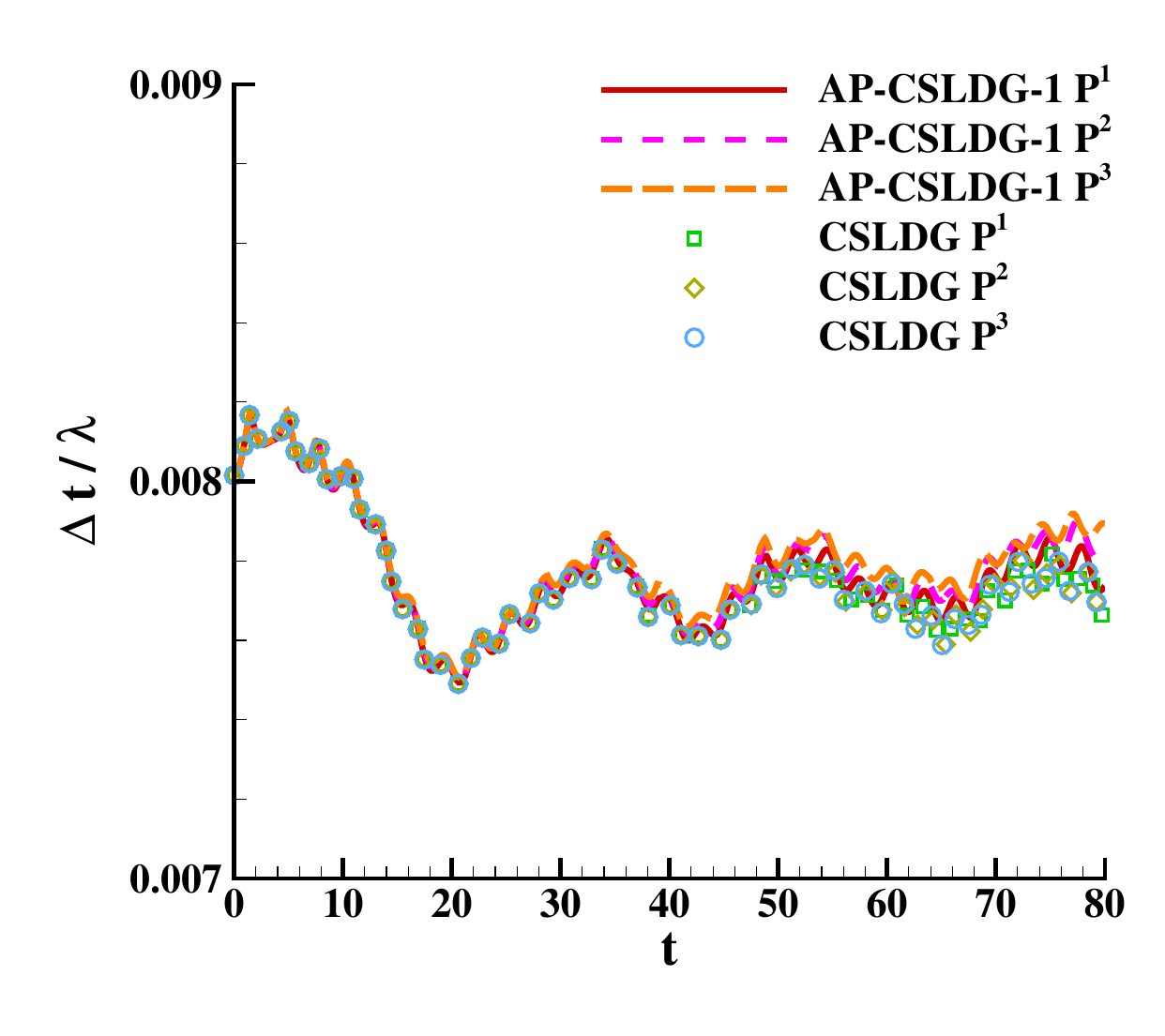}
    \end{minipage}
    \caption{Bump-on-tail instability: In non-quasi-neutral regime ($\lambda=1$), time evolution of the logarithm of electrostatic energy (left) and the ratio of time step to Debye length (right) simulated by AP-CSLDG-1 scheme and CSLDG scheme.} 
    \label{fig:ex5lambda1}
\end{figure}
\begin{figure}[h!]
    \centering    
    \begin{minipage}[t]{0.48\linewidth}
        \centering
        \includegraphics[width=\linewidth]{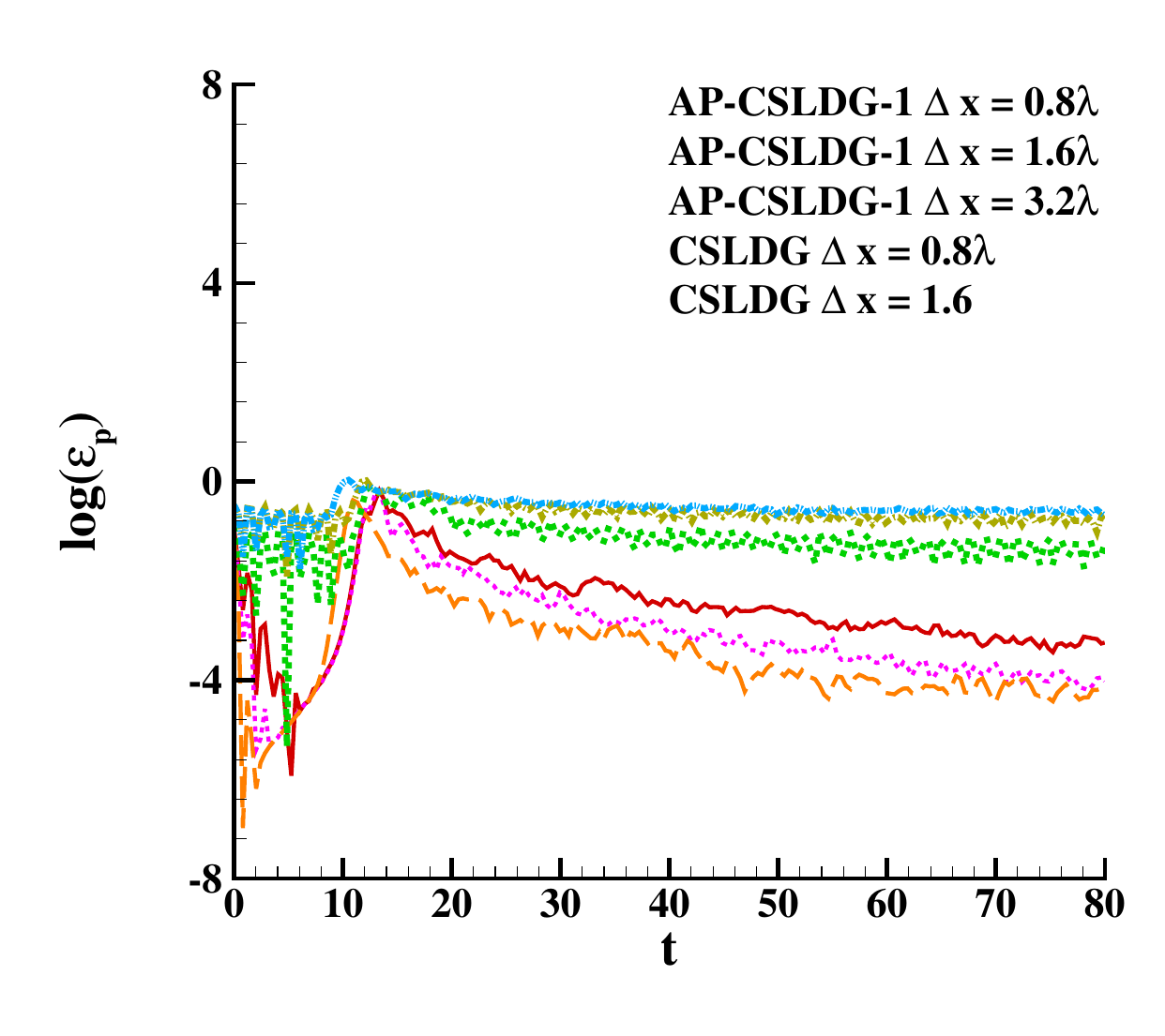}
    \end{minipage}
    \hfill
    \begin{minipage}[t]{0.48\linewidth}
        \centering
        \includegraphics[width=\linewidth]{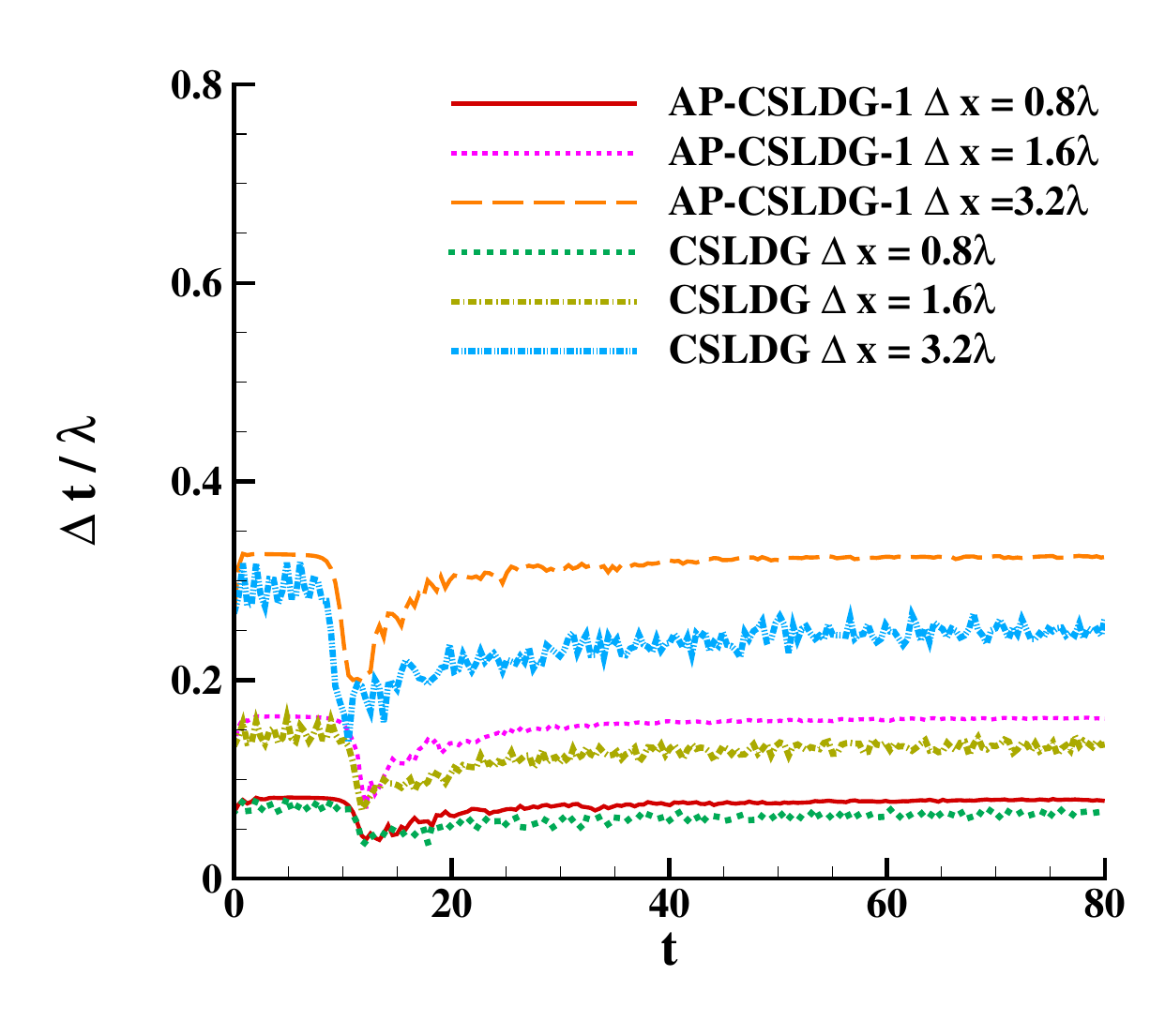}
    \end{minipage}
    \caption{Bump-on-tail instability: When $\lambda=0.1$, time evolution of the logarithm of electrostatic energy (left) and the ratio of time step to Debye length (right) simulated by AP-CSLDG-1 scheme and CSLDG scheme.} 
    \label{fig:ex5lambda01_gd}
\end{figure}
\begin{figure}[h!]
    \centering    
    \begin{minipage}[t]{0.48\linewidth}
        \centering
        \includegraphics[width=\linewidth]{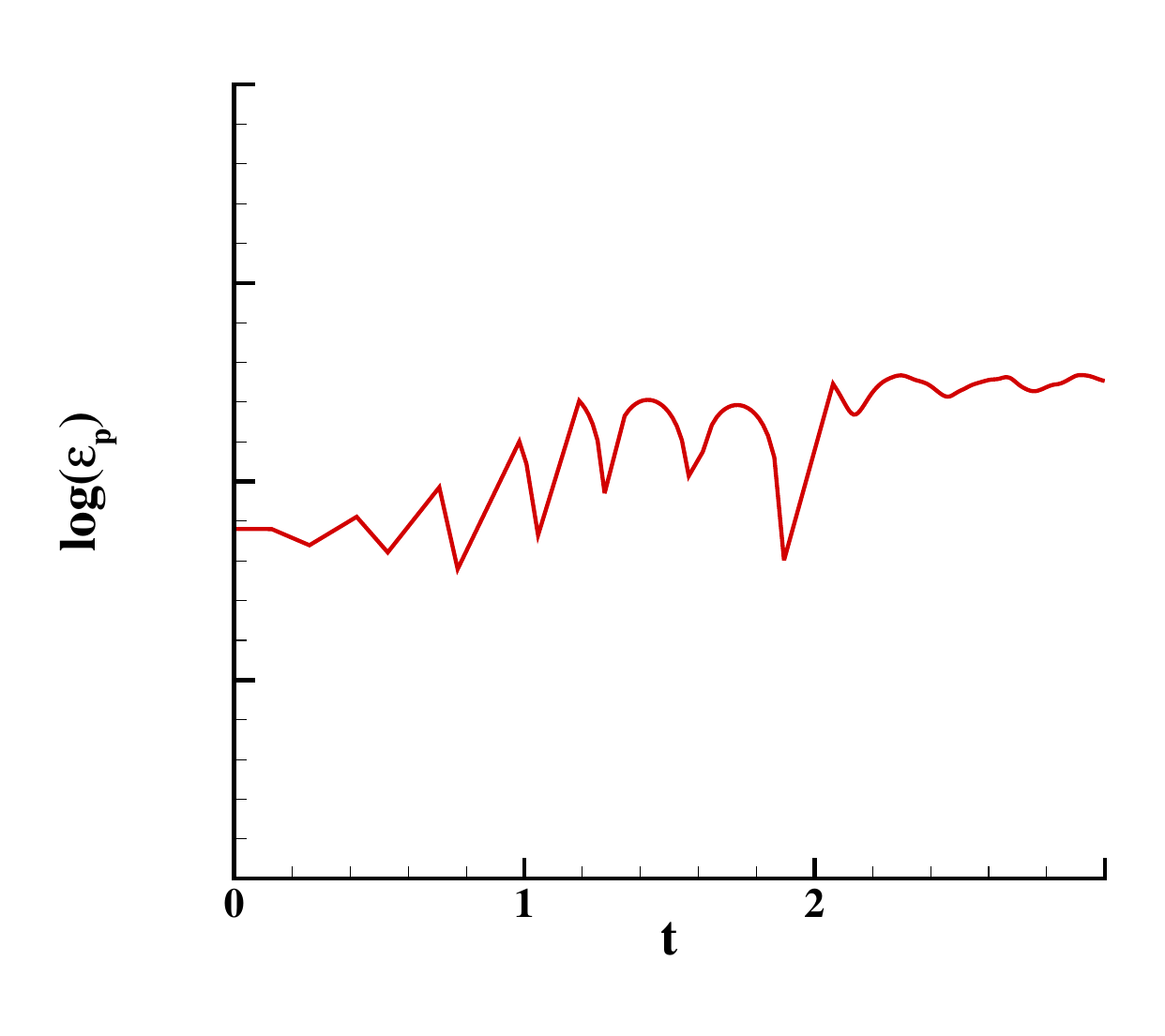}
    \end{minipage}
    \hfill
    \begin{minipage}[t]{0.48\linewidth}
        \centering
        \includegraphics[width=\linewidth]{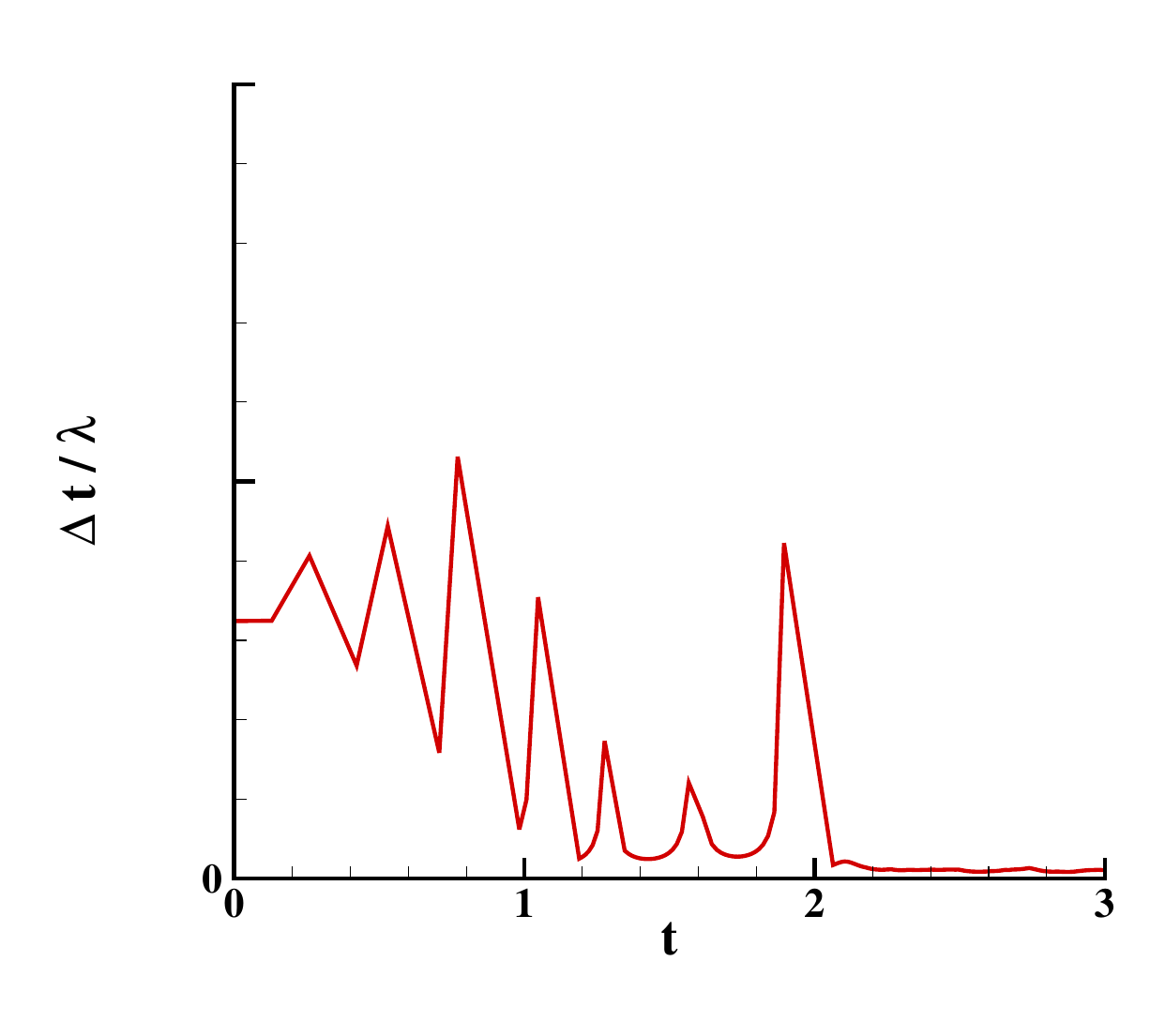}
    \end{minipage}
    \caption{Bump-on-tail instability: When $\lambda=0.1$, time evolution of the logarithm of electrostatic energy (left) and the ratio of time step to Debye length (right) simulated by CSLDG scheme.}
    \label{fig:ex5lambda01_bd}
\end{figure}
\begin{figure}[h!]
    \centering    
    \begin{minipage}[t]{0.48\linewidth}
        \centering
        \includegraphics[width=\linewidth]{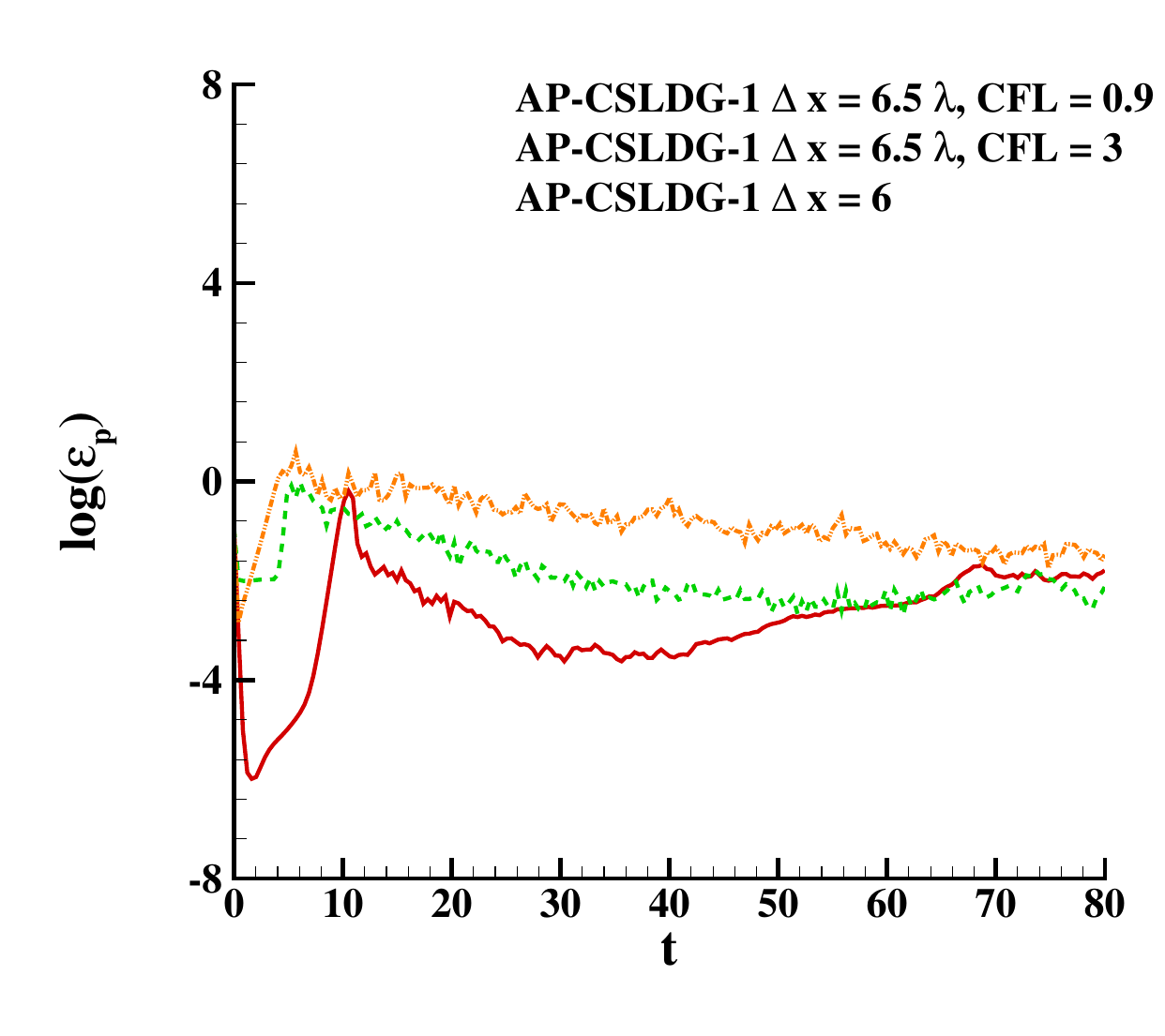}
    \end{minipage}
    \hfill
    \begin{minipage}[t]{0.48\linewidth}
        \centering
        \includegraphics[width=\linewidth]{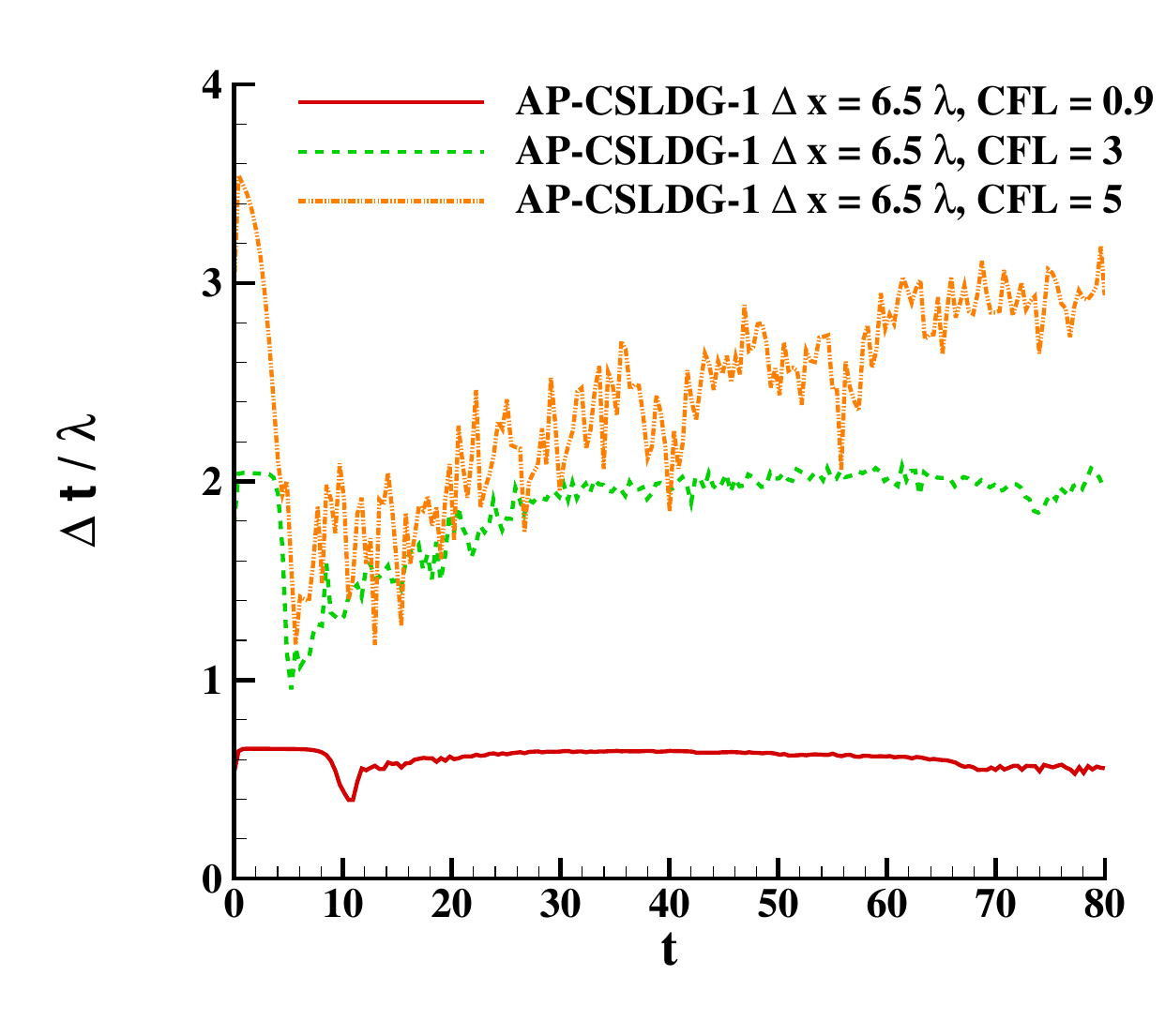}
    \end{minipage}
    \caption{Bump-on-tail instability: When $\lambda=0.1$, time evolution of the logarithm of electrostatic energy (left) and the ratio of time step to Debye length (right) simulated by AP-CSLDG-1 scheme.} 
    \label{fig:ex5lambda01_gdd}
\end{figure}
\begin{figure}[h!]
    \centering    
    \begin{minipage}[t]{0.48\linewidth}
        \centering
        \includegraphics[width=\linewidth]{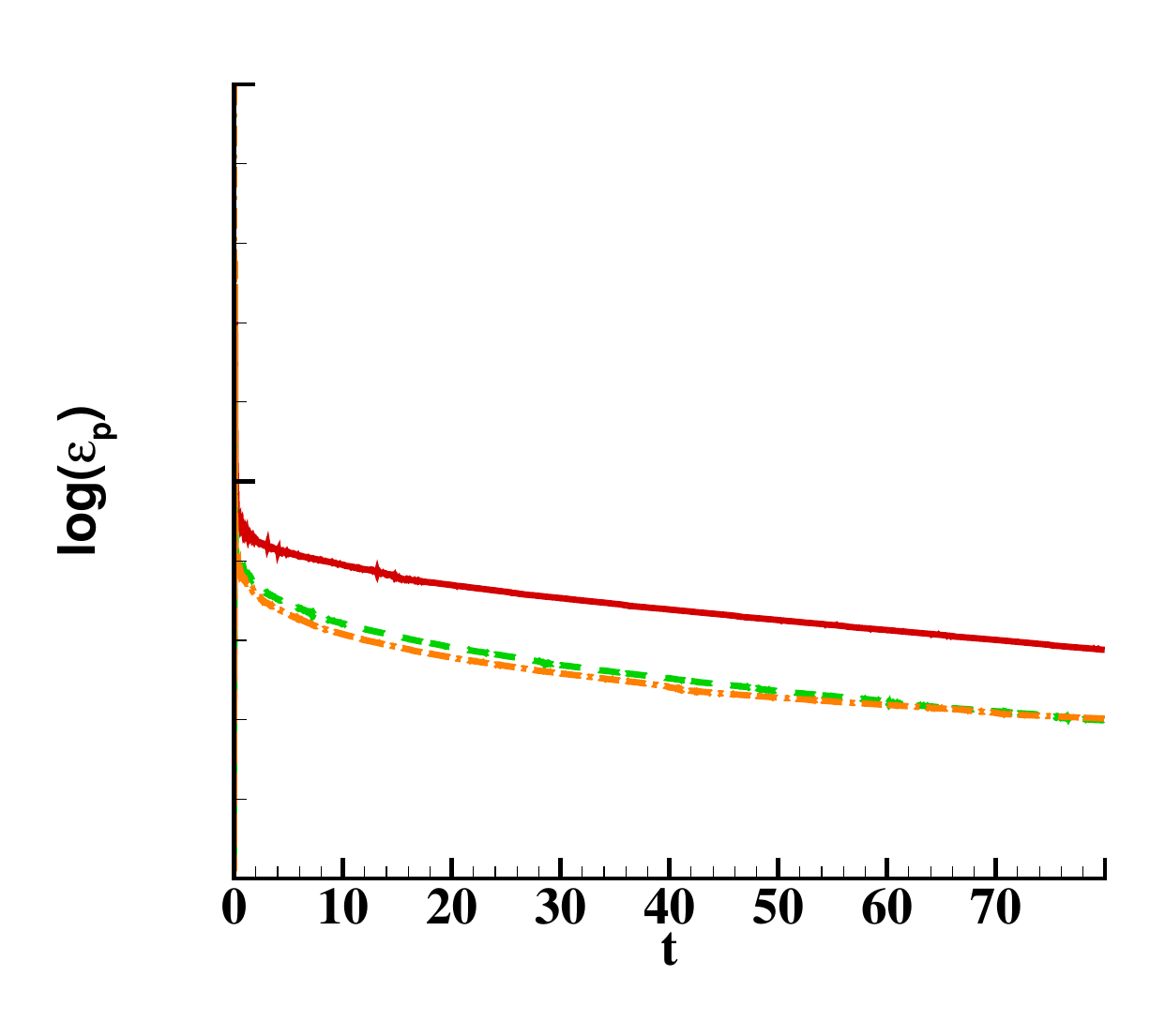}
    \end{minipage}
    \hfill
    \begin{minipage}[t]{0.48\linewidth}
        \centering
        \includegraphics[width=\linewidth]{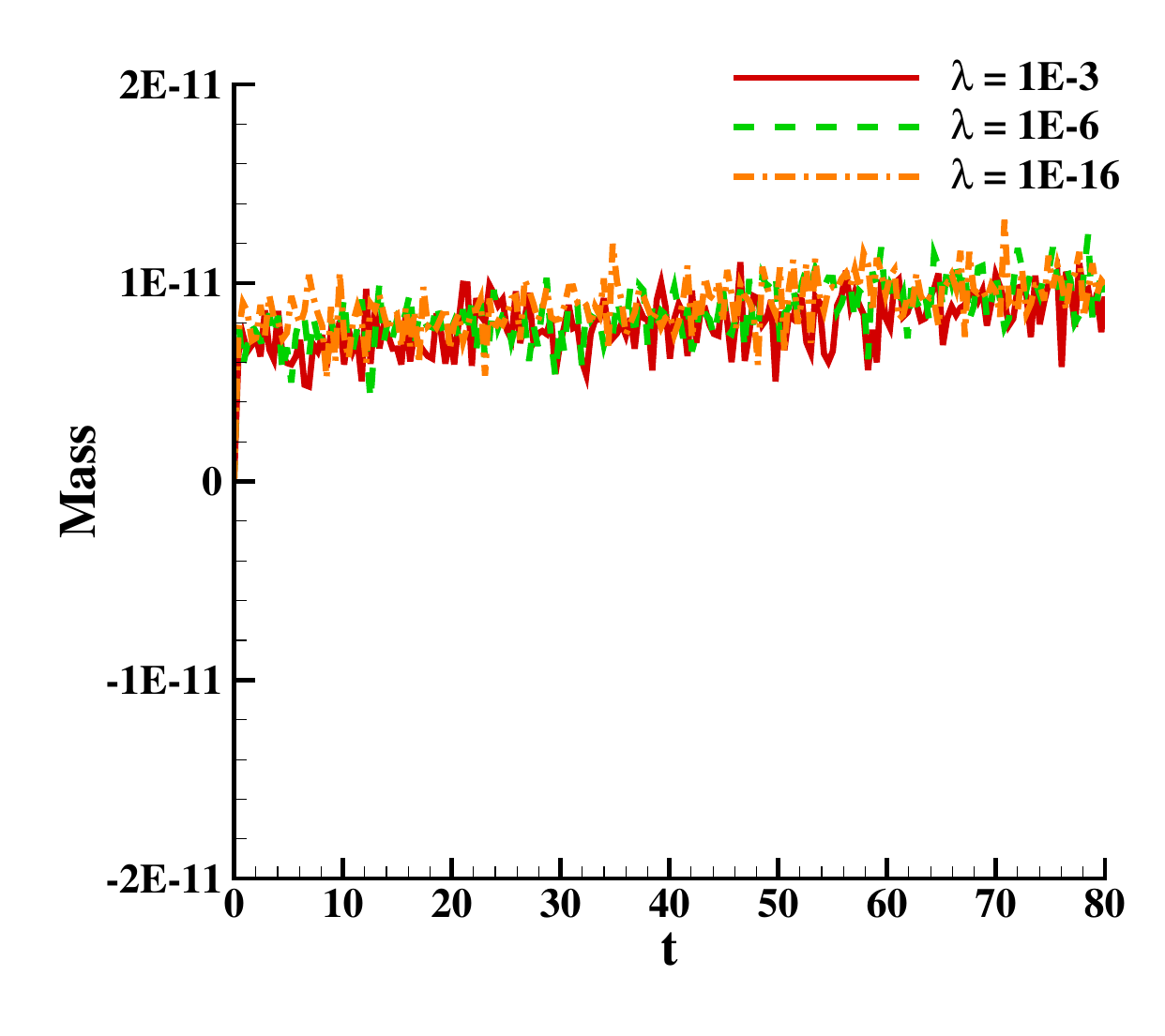}
    \end{minipage}
    \caption{Bump-on-tail instability: When $\lambda\to 0$, time evolution of the logarithm of electrostatic energy (left) and the relative deviations for mass (right) simulated by AP-CSLDG-1 scheme.}
    \label{fig:ex5quasi}
\end{figure}

\section{Conclusions}

\hspace{4pt}
We have presented asymptotic‑preserving conservative semi‑Lagrangian discontinuous Galerkin (AP‑CSLDG) schemes for the Vlasov–Poisson (VP) system in the quasi‑neutral limit. The proposed approach combines (i) a reformulation of the VP system that remains well defined in the quasi‑neutral regime, (ii) a semi‑Lagrangian treatment of the Vlasov equation, (iii) a discontinuous Galerkin spatial discretization, and (iv) an appropriate operator‑splitting strategy. These ingredients enable the proposed schemes to provide reliable results in both non‑quasi‑neutral and quasi‑neutral regimes, even when the grid spacing and time step do not resolve the Debye length. The conservative semi‑Lagrangian solver updates numerical solutions in a stable and efficient manner, while exactly conserving the total particle number. The high‑order DG discretization reliably captures fine‑scale plasma dynamics on relatively coarse meshes. We have also performed theoretical analysis of our schemes. Their asymptotic‑preserving properties are certified by rigorous proofs of asymptotic consistency and stability.

Importantly, the AP‑CSLDG splitting schemes offer a per‑step computational cost comparable to that of explicit methods. The above properties of our approach makes it an attractive practical choice for multiscale plasma simulations. Moreover, the AP‑CSLDG framework is well suited for extensions to more realistic plasma models, including electromagnetic (Vlasov–Maxwell or Vlasov–Darwin) and collisional (BGK or Fokker–Planck) systems. Additionally, unsplit semi-Lagrangian methods \cite{Cai2018,sun2025fourth} could be explored to eliminate the need for operator splitting in multi-dimensional simulations.

\newpage
\textbf{Data availability} The datasets generated during and/or analyzed during the current study are available from
 the corresponding author upon reasonable request.

\section{Declarations}
The authors declared that they have no conflicts of interest to this work. 
We declare that we do not have any commercial or associative interest that represents a conflict of interest in connection with the work submitted.
\bibliographystyle{unsrt}
\bibliography{ref}

\end{document}